\documentclass[pdflatex,sn-mathphys-num]{sn-jnl}

\usepackage{graphicx}%
\usepackage{multirow}%
\usepackage{amsmath,amssymb,amsfonts}%
\usepackage{amsthm}%
\usepackage{mathrsfs}%
\usepackage{stmaryrd}
\usepackage[title]{appendix}%
\usepackage{xcolor}%
\usepackage{textcomp}%
\usepackage{manyfoot}%
\usepackage{booktabs}%
\usepackage{algorithm}%
\usepackage{algorithmicx}%
\usepackage{algpseudocode}%
\usepackage{listings}%
\usepackage{cleveref}
\usepackage{bm}%
\usepackage{cancel}

\theoremstyle{thmstyleone}%
\newtheorem{theorem}{Theorem}
\newtheorem{lemma}{Lemma}%
\newtheorem{assumption}{Assumption}%
\theoremstyle{thmstyletwo}%

\theoremstyle{thmstylethree}%

\begin{document}


\title{A Unified Model for Thermo- and Multiple-Network Poroelasticity with a Global-in-Time Iterative Decoupling Scheme}


\author[1]{\fnm{Huipeng} \sur{Gu}}\email{guhuipeng@suit-sz.edu.cn}

\author*[2]{\fnm{Mingchao} \sur{Cai}}\email{mingchao.cai@morgan.edu}

\author[3]{\fnm{Jingzhi} \sur{Li}}\email{li.jz@sustech.edu.cn}

\author[1]{\fnm{Yu} \sur{Jiang}}\email{yujiang686@163.com}

\affil[1]{\orgdiv{School of Humanities and Fundamental Sciences}, \orgname{Shenzhen University of Information Technology}, \orgaddress{\street{2188 Longxiang Avenue}, \city{Shenzhen}, \postcode{518172}, \state{Guangdong}, \country{P.R. China}}}

\affil[2]{\orgdiv{Department of Mathematics}, \orgname{Morgan State University}, \orgaddress{\street{1700 East Cold Spring Lane}, \city{Baltimore}, \postcode{21251}, \state{Maryland}, \country{USA}}}

\affil[3]{\orgdiv{Department of Mathematics}, \orgname{Southern University of Science and Technology}, \orgaddress{\street{1088 Xueyuan Avenue}, \city{Shenzhen}, \postcode{518055}, \state{Guangdong}, \country{P.R. China}}}


\abstract{
This paper introduces a unified model for thermo-poroelasticity and multiple-network poroelasticity, reformulated into a total-pressure-based system. We first establish the well-posedness of the problem via a Galerkin-based argument and subsequently introduce a robust space-time finite element approximation. To efficiently solve the fully coupled system, we propose a global-in-time iterative algorithm that sequentially decouples the mechanics from the transport equations, while incorporating necessary stabilization terms. We explicitly analyze the convergence rate and provide a rigorous proof that the proposed scheme constitutes a contraction mapping under physically relevant conditions, thereby ensuring its unconditional convergence. Numerical experiments confirm the theoretical stability bounds and demonstrate optimal convergence rates in both space and time, yielding solutions free of non-physical pressure oscillations.
}


\maketitle

\section{Introduction}\label{sec1}

Poroelasticity provides a theoretical framework that bridges elasticity and fluid mechanics, describing the interaction between solid deformation and pore fluid pressure. Building on Terzaghi’s one-dimensional consolidation theory \cite{terzaghi1943theoretical}, Biot extended the framework to three dimensions and established the general theory of dynamic poroelasticity \cite{biot1935problemc,biot1941general}, which is now widely used in applications ranging from geotechnical engineering to biomechanics \cite{fu2021thermo, olabanjo2025finite}. More recently, mathematical models have been expanded to capture complex coupled behaviors, such as linear thermo-poroelasticity \cite{yi2024physics,cai2025efficient} and multiple-network poroelasticity \cite{lee2019mixed,zhao2025optimally} (e.g., the Barenblatt-Biot model for dual-porosity systems \cite{barenblatt1960basic}). Notably, recent efforts have also focused on unifying these complex interactions, establishing common analytical frameworks for multiphysics problems such as thermo/poro-viscoelasticity \cite{bonetti2025unified} and the coupled dynamics of thermoelastic diffusion and thermo-poroelasticity \cite{nataraj2025unified}. Identifying the shared mathematical structure of these systems, this work proposes a unified framework applicable to both linear thermo-poroelasticity and dual-porosity models, thereby streamlining analysis and numerical treatment.

Ensuring numerical stability and robustness for poroelastic systems remains challenging, particularly in incompressible or low-permeability regimes where classical displacement–pressure formulations often suffer from nonphysical oscillations and volumetric locking \cite{phillips2007coupling,phillips2009overcoming}. To mitigate these difficulties, Oyarzúa et al. proposed and analyzed reformulations based on a total pressure variable. \cite{oyarzua2016locking} and Lee et al. \cite{lee2017parameter}. This approach has become a widely adopted framework because it effectively avoids locking while remaining compatible with standard mixed finite element spaces, for example, using Taylor–Hood elements for the displacement and total pressure together with Lagrange elements for the pore pressure. This total-pressure formulation has subsequently been extended to several complex multiphysics settings, including multi-network poroelasticity \cite{zhao2025optimally}, frequency-domain poroelasticity \cite{carcamo2024stabilized}, Biot–Stokes coupling \cite{ruiz2022biot}, and Biot–Brinkman models \cite{cui2025locking}. 
These developments highlight both the theoretical and practical versatility of the formulation, demonstrating its flexibility in addressing increasingly sophisticated coupled processes in porous media.

Beyond stable formulations, the development of efficient numerical solvers is essential for large-scale simulations of time-dependent multiphysics systems. Monolithic approaches, although theoretically robust, often lead to extremely large coupled algebraic systems whose solution cost grows rapidly with problem size. To alleviate this difficulty, several strategies have been explored. One direction focuses on designing advanced preconditioners to accelerate the convergence of monolithic solvers \cite{chen2020robust,piersanti2021parameter}. Another direction aims to decouple the multiphysics processes using operator-splitting techniques \cite{chaabane2018splitting,altmann2024semi}. However, classical splitting schemes frequently impose restrictive time-step constraints depending on physical parameters, and traditional domain decomposition strategies remain sequential in time, limiting their scalability for long-time simulations.

To combine the stability of monolithic formulations with the computational efficiency of decoupled methods, iterative decoupling strategies have been proposed for poroelastic systems \cite{kim2011stability,mikelic2013convergence}. These methods apply fixed-point iterations between the mechanics and transport subproblems, thereby recovering the accuracy of the fully coupled solution while allowing the use of specialized solvers for each subsystem. Classical implementations typically proceed in a time-marching fashion \cite{dana2018multiscale,storvik2019optimization,almani2023convergence}. More recently, global-in-time iterative formulations have been introduced \cite{borregales2019partially,ahmed2020adaptive}, enabling one of the subproblems to be solved in parallel across the temporal domain and thus improving computational scalability. While these ideas have been primarily developed for the classical Biot model, their modular structure naturally lends itself to extensions toward more complex coupled systems such as thermo-poroelasticity and multi-network poroelasticity.

Motivated by these developments, this paper proposes a unified mathematical and computational framework for thermo- and multiple-network poroelasticity. In particular, we introduce a unified four-field mixed formulation based on the total pressure variable that can encompass both linear thermo-poroelasticity and dual-porosity models within a single framework. We first establish the well-posedness of the weak formulation using a Galerkin-based argument and then develop a robust space–time iterative solution strategy. The proposed approach decomposes the coupled system into two subproblems—a transport system and a mechanics system—which are solved sequentially in an alternating manner. This strategy is implemented through a fully discrete global-in-time algorithm that exploits the quasi-static nature of the mechanics subsystem, enabling a partially “parallel-in-time” implementation \cite{borregales2019partially}. Stabilization terms are incorporated to control nonphysical oscillations in the transport equations \cite{rodrigo2016stability}. We further provide a rigorous convergence analysis showing that the resulting scheme constitutes a contraction mapping under physically relevant conditions. Consequently, the proposed framework offers a unified, stable, and provably robust approach for simulating coupled processes in porous media.

The remainder of the paper is organized as follows. In Section \ref{sec:2}, we present the governing equations of the unified framework, the model reformulations, the corresponding weak formulation, and its iterative approximation. Section \ref{sec:3} details the fully discrete finite element approximations and the proposed global-in-time iterative decoupled algorithm, including a rigorous convergence analysis. Finally, Section \ref{sec:4} presents numerical experiments that verify the theoretical convergence rates, followed by concluding remarks in Section \ref{sec:5}.

\section{Mathematical model}
\label{sec:2}

Let $\Omega \subset \mathbb{R}^d$ ($d \in \{2, 3\}$) be an open and bounded domain with a Lipschitz continuous boundary $\Gamma := \partial\Omega$, and let $(0, T_f]$ be the time interval with $T_f > 0$. We denote the partial time derivative of a function $u$ by $\partial_t {u}$. Let $L^2(\Omega)$ be the space of square-integrable functions on $\Omega$, equipped with the standard inner product $(\cdot, \cdot)$ and norm $\|\cdot\|$. We denote by $H^1(\Omega)$ the Sobolev space of functions with square-integrable first derivatives, equipped with the norm $\|\cdot\|_{H^1}$, and by $H^1_0(\Omega)$ the subspace of functions in $H^1(\Omega)$ with zero trace on $\Gamma$. For a Banach space $X$ and $1 \le s \le \infty$, the Bochner space $L^s(0, T_f; X)$ consists of measurable functions $v: (0, T_f) \to X$ such that the norm $\|v\|_{L^s(0, T_f;X)}$ is finite, where $\|v\|_{L^s(0, T_f;X)} := ( \int_0^{T_f} \|v(t)\|_X^s \, dt )^{1/s}$ for $1 \le s < \infty$, and $\|v\|_{L^\infty(0, T_f;X)} := \operatorname*{ess\,sup}_{t \in (0, T_f)} \|v(t)\|_X$. Moreover, for a Hillbert space $X$, we define the Bochner-Sobolev space $H^1(0, T_f; X) := \{ v \in L^2(0, T_f; X) : \partial_t v \in L^2(0, T_f; X) \}$, which is endowed with the norm $\|v\|_{H^1(0, T_f; X)} := \big( \|v\|_{L^2(0, T_f; X)}^2 + \|\partial_t v\|_{L^2(0, T_f; X)}^2 \big)^{1/2} $.

The unified three-field model for coupled mechanics and transport processes seeks the solid displacement vector $\bm{u}$ and two generalized pressure variables $\phi$ and $\psi$ satisfying the following system:
\begin{subequations}
\label{threefield}
\begin{align} 
-\nabla\cdot\sigma(\bm{u}) + \alpha \nabla \phi + \beta \nabla \psi &= \bm{f}& \text{in} \ \Omega \times (0,T_f], 
\\
c_1\partial_t {\phi} - b_0\partial_t {\psi} + \alpha \nabla\cdot\partial_t {\bm{u}} - \nabla\cdot( \mathbf{K} \nabla \phi) + \gamma (\phi - \psi) &= g& \text{in} \ \Omega \times (0,T_f], 
\\ 
c_2\partial_t {\psi} - b_0\partial_t {\phi} + \beta \nabla\cdot\partial_t {\bm{u}} - \nabla\cdot( \mathbf{D} \nabla \psi) + \gamma (\psi - \phi) &= h& \text{in} \ \Omega \times (0,T_f].
\end{align} 
\end{subequations}
Here, the stress tensor is defined by 
\begin{equation*}
\sigma(\bm{u}) = 2\mu\varepsilon(\bm{u}) + \lambda(\nabla\cdot \bm{u})\mathbf{I}, \quad \text{with the strain tensor} \quad \varepsilon(\bm{u}) = \tfrac{1}{2} ( \nabla \bm{u} + \nabla \bm{u}^T ),
\end{equation*}
where $\mu$ and $\lambda$ are the Lamé coefficients, related to Young's modulus $E$ and Poisson's ratio $\nu$ via $\lambda = \frac{E\nu}{(1+\nu)(1-2\nu)}$ and $\mu = \frac{E}{2(1+\nu)}$, $\mathbf{I}$ is the identity matrix.
The system is driven by the body force $\bm{f}$ and source terms $g$ and $h$. Regarding the physical parameters: $\alpha$ and $\beta$ denote the fluid-solid coupling coefficients; $c_1$ and $c_2$ are the specific storage coefficients; $b_0$ is the dilation coefficient; $\mathbf{K}$, $\mathbf{D}$ are the permeability (or conductivities) tensors; and $\gamma$ is a transfer coefficient that governs the exchange rate. 

To close the system \eqref{threefield}, we partition the boundary $\Gamma$ into two disjoint parts $\Gamma_D$ and $\Gamma_N$, such that $\Gamma = \overline{\Gamma_D} \cup \overline{\Gamma_N}$, assuming $|\Gamma_D| > 0$. We consider the following mixed-type boundary conditions:
\begin{subequations}
\label{boundaries}
\begin{align}
    \bm{u} &= \bm{0} &&\text{on } \Gamma_D \times (0, T_f], \label{bc:u_dirichlet} \\
    [\sigma(\bm{u}) - (\alpha\phi + \beta\psi)\mathbf I]\bm{n} &= \bm{0} &&\text{on } \Gamma_N \times (0, T_f], \label{bc:u_neumann} \\
    \phi = 0, \quad \psi & = 0 &&\text{on } \Gamma \times (0, T_f], \label{bc:p}
\end{align}
\end{subequations}
where $\bm n$ represents the unit outward normal, along with the initial conditions:
\begin{equation}
    \bm{u}(0) = \bm{u}_0, \quad \phi(0) = \phi_0, \quad \psi(0) = \psi_0 \quad \text{in } \Omega. \label{ic}
\end{equation}


\subsection{Model specializations and reformulation}
The general system \eqref{threefield} can be specialized to describe several important physical phenomena. We present two prominent examples below.

\subsubsection*{Linear thermo-poroelasticity model}

A fully coupled linear thermo-poroelasticity model is recovered from the general system by identifying the generalized variables as the pore pressure, \(p := \phi\), and temperature, \(T:= \psi\). In this case, the transfer coefficient is set to \(\gamma = 0\), indicating no direct coupling between the pressure and temperature fields. The system then reduces to:
\begin{align*}
-\nabla\cdot\sigma(\bm{u}) + \alpha \nabla p + \beta \nabla T &= \bm{f}, 
\\
c_1\partial_t {p} - b_0\partial_t {T} + \alpha \nabla\cdot\partial_t {\bm{u}} - \nabla\cdot( \mathbf{K} \nabla p) &= g, 
\\ 
c_2\partial_t {T} - b_0\partial_t {p} + \beta \nabla\cdot\partial_t {\bm{u}} - \nabla\cdot( \mathbf{D} \nabla T) &= h. 
\end{align*}

\subsubsection*{Barenblatt-Biot model}

The framework also encapsulates the Barenblatt-Biot dual-porosity model, a special case of quasi-static Multiple-Network Poroelasticity (MPET). This is recovered by interpreting both $\phi$ and $\psi$ as pressure fields in different porous networks, denoted by $p_1 := \phi$ and $p_2 := \psi$. A key assumption for this model is the absence of cross-coupling in the storage terms, achieved by setting the dilation coefficient $b_0 = 0$. Under this assumption, the governing equations become: 
\begin{align*} 
-\nabla\cdot\sigma(\bm{u}) + \alpha \nabla p_1 + \beta \nabla p_2 &= \bm{f},
\\ 
c_1\partial_t {p_1} + \alpha \nabla\cdot\partial_t {\bm{u}} - \nabla\cdot(\mathbf{K} \nabla p_1) + \gamma(p_1 - p_2) &= g, 
\\ 
c_2\partial_t {p_2} + \beta \nabla\cdot\partial_t {\bm{u}} - \nabla\cdot(\mathbf{D} \nabla p_2) + \gamma(p_2 - p_1) &= h. 
\end{align*}

\subsubsection*{Four-field reformulation}

Following the approach in \cite{oyarzua2016locking} to obtain stable numerical approximations, particularly in the nearly incompressible limit, we introduce an auxiliary variable representing the total pressure, denoted by $\xi$, and defined as:
\begin{align*}
    \xi = -\lambda \nabla\cdot\bm{u} + \alpha \phi + \beta \psi.
\end{align*}
Substituting this new unknown into the original three-field system \eqref{threefield} yields the following four-field reformulation on $\Omega \times (0,T_f]$:
\begin{subequations}
\label{fourfield}
\begin{align} 
-\nabla\cdot( 2\mu\varepsilon(\bm{u}) - \xi\mathbf{I}) &= \bm{f}, 
\\ 
\nabla\cdot\bm{u} + \tfrac{1}{\lambda} \xi - \tfrac{\alpha}{\lambda} \phi - \tfrac{\beta}{\lambda} \psi &= 0, 
\\ 
\big(c_1 + \tfrac{\alpha^2}{\lambda}\big) \partial_t {\phi} + \big( \tfrac{\alpha \beta}{\lambda} - b_0 \big)\partial_t {\psi} - \tfrac{\alpha}{\lambda} \partial_t {\xi} - \nabla\cdot(\mathbf{K} \nabla \phi) + \gamma(\phi - \psi) &= g, 
\\
\big(c_2 + \tfrac{\beta^2}{\lambda}\big) \partial_t {\psi} + \big( \tfrac{\alpha \beta}{\lambda} - b_0 \big)\partial_t {\phi} - \tfrac{\beta}{\lambda} \partial_t {\xi} - \nabla\cdot(\mathbf{D} \nabla \psi) + \gamma(\psi - \phi) &= h.
\end{align}
\end{subequations}
The boundary conditions \eqref{boundaries} and initial conditions \eqref{ic} are adapted to the system \eqref{fourfield} by modifying the Neumann condition \eqref{bc:u_neumann} to read 
\begin{align}
    (2\mu\varepsilon(\bm{u}) - \xi\mathbf I)\bm{n} = \bm{0} \quad \text{on} \quad \Gamma_N \times (0, T_f].
\end{align}
We note that no explicit boundary condition is imposed on the intermediate variable $\xi$, while its initial condition is defined as $\xi(0) = \xi_0:= -\lambda \nabla\cdot\bm{u}_0 + \alpha \phi_0 + \beta \psi_0$.

\subsection{Weak formulation} 

We introduce the following function spaces:
\begin{align*}
    \bm{V} := [H^1_{\Gamma_D}(\Omega)]^d, \qquad
    W := L^2(\Omega), \qquad 
    Q = S := H^1_0(\Omega),
\end{align*}
where $H^1_{\Gamma_D}(\Omega)$ denotes the space of functions in $H^1(\Omega)$ with zero trace on $\Gamma_D$. We recall that the pair $(\bm{V}, W)$ satisfies the continuous inf-sup condition \cite{brenner1993nonconforming}: there exists a constant $C_{is} > 0$ such that
\begin{align}
    C_{is} \| w \|_{L^2(\Omega)}
    \le
	\sup_{\bm{0} \not= \bm{v}\in \bm{V}} \frac{  ( \nabla \cdot \bm{v}, w) }{\| \bm{v} \|_{H^1}}  \quad \forall w \in W. \label{infsupcon}
\end{align}
Furthermore, the coercivity of the elastic bilinear form is guaranteed by Korn's inequality \cite{nitsche1981korn}, which ensures the existence of a constant $C_K(\Omega,\Gamma_D) > 0$ such that 
\begin{align} 
    \| \bm{v} \|_{H^1} \leq C_K \| \varepsilon(\bm{v}) \|  \quad \forall \bm{v} \in \bm{V}.  \label{korn}
\end{align}

\begin{assumption}
    \label{assumption1}
    For the subsequent analysis, we make the following assumptions regarding the data and physical parameters:
    \begin{enumerate}
        \item \textbf{Isotropy:} The permeabilities (or conductivities) tensors are isotropic, such that $\mathbf{K} = K \mathbf{I}$ and $\mathbf{D} = D \mathbf{I}$. 
        
        \item \textbf{Coefficients:} The parameters $\alpha, \beta, \lambda, \mu, K$, and $D$ are strictly positive constants. The coefficients $c_1, c_2, b_0$, and $\gamma$ are non-negative constants, satisfying the structural constraints $c_1 - b_0 \ge 0$ and $c_2 - b_0 \ge 0$. 
        
        \item \textbf{Regularity:} The initial conditions are sufficiently regular, specifically $\bm{u}_0 \in \bm{V}$, $\phi_0 \in Q$ and $\psi_0 \in S$. The forcing terms are assumed to be sufficiently regular in time, with $\bm{f} \in H^1(0, T_f;[L^2(\Omega)]^d)$ and $g, h \in H^1(0, T_f; L^2(\Omega))$.
    \end{enumerate}
\end{assumption}

Multiplying the equations in system \eqref{fourfield} by test functions, applying integration by parts, and imposing the boundary conditions, we obtain the weak formulation: For almost every $t \in (0, T_f]$, find $(\bm{u}(t), \xi(t), \phi(t), \psi(t)) \in \bm{V} \times W \times Q \times S$ such that
\begin{subequations}
\label{weakform}
\begin{align}
2\mu(\varepsilon(\bm u),\varepsilon(\bm v))
- (\xi,\nabla\cdot\bm v)
& =  (\bm f,\bm v) \quad & \forall & \bm v \in \bm V, \label{wf:momentum}
\\
(\nabla\cdot\bm u, w)
+\tfrac{1}{\lambda} (\xi, w)
-\tfrac{\alpha}{\lambda} (\phi, w)
-\tfrac{\beta}{\lambda} (\psi, w)
& = 0 \quad & \forall & w \in W,  \label{wf:constraint}
\\
(c_1+\tfrac{\alpha^2}{\lambda})(\partial_t {\phi}, q)
+(\tfrac{\alpha\beta}{\lambda}-b_0)(\partial_t {\psi}, q)
-\tfrac{\alpha}{\lambda}(\partial_t {\xi}, q) \nonumber \\
+ K(\nabla\phi,\nabla q) + \gamma (\phi-\psi, q)
& = (g,q) \quad & \forall & q \in Q, \label{wf:phi}
\\
(c_2+\tfrac{\beta^2}{\lambda})(\partial_t {\psi}, s)
+(\tfrac{\alpha\beta}{\lambda}-b_0)(\partial_t {\phi}, s)
-\tfrac{\beta}{\lambda}(\partial_t {\xi}, s) \nonumber \\
+ D (\nabla\psi,\nabla s) + \gamma (\psi-\phi, s)
& = (h,s) \quad & \forall & s \in S. \label{wf:psi}
\end{align}
\end{subequations}

\begin{theorem}[Well-posedness]
\label{thm:wp}
Under Assumption \ref{assumption1},
System \eqref{weakform} admits a unique solution
\[
(\bm u,\xi,\phi,\psi)\in
H^1(0,T_f;\bm V)\times H^1(0,T_f;W)\times
H^1(0,T_f;Q)\times H^1(0,T_f;S).
\]
\end{theorem}
The proof of this result relies on a Galerkin approximation combined with a priori energy estimates and arguments from the theory of differential algebraic equations (DAE). Related analytical techniques for coupled systems of similar structure can be found in \cite{phillips2008coupling, yi2014convergence, brun2019well, ahmed2019adaptive}. 

First, we construct the finite-dimensional approximations as follows.
Let $\{\bm v_i\}_{i\in\mathbb N}$, $\{w_i\}_{i\in\mathbb N}$,
$\{q_i\}_{i\in\mathbb N}$, and $\{s_i\}_{i\in\mathbb N}$
be Hilbert bases of the spaces $\bm V$, $W$, $Q$, and $S$, respectively.
For fixed strictly positive integers $j,k,l,$ and $m$, we define the finite-dimensional subspaces
\begin{align*}
    \bm{V}_j &:= \mathrm{span}\{\bm{v}_1,\dots,\bm{v}_j\} \subset \bm{V}, &
    W_k &:= \mathrm{span}\{w_1,\dots,w_k\} \subset W, \\
    Q_l &:= \mathrm{span}\{q_1,\dots,q_l\} \subset Q, &
    S_m &:= \mathrm{span}\{s_1,\dots,s_m\} \subset S.
\end{align*}
Let $(\bm{u}_j, \xi_k, \phi_l, \psi_m) : (0, T_f] \to \bm{V}_j \times W_k \times Q_l \times S_m$ denote the solution satisfying the following variational system for almost every $t \in (0, T_f]$:

\begin{subequations} 
\label{eq:discrete_weak_form} 
\begin{align} 
2\mu(\varepsilon(\bm{u}_j), \varepsilon(\bm{v}_i)) - (\xi_k, \nabla\cdot\bm{v}_i) &= (\bm{f}, \bm{v}_i) && \forall i=1,\dots,j, 
\\ 
(\nabla\cdot\bm{u}_j, w_i) + \tfrac{1}{\lambda}(\xi_k, w_i) - \tfrac{\alpha}{\lambda}(\phi_l, w_i) - \tfrac{\beta}{\lambda}(\psi_m, w_i) &= 0 && \forall i=1,\dots,k, 
\\
(c_1+\tfrac{\alpha^2}{\lambda})(\partial_t {\phi}_l, q_i) + (\tfrac{\alpha\beta}{\lambda}-b_0)(\partial_t {\psi}_m, q_i) - \tfrac{\alpha}{\lambda}(\partial_t {\xi}_k &, q_i) \nonumber \\
+ K ( \nabla\phi_l, \nabla q_i) + \gamma(\phi_l-\psi_m, q_i) &= (g, q_i) && \forall i=1,\dots,l, 
\\ 
(c_2+\tfrac{\beta^2}{\lambda})(\partial_t {\psi}_m, s_i) + (\tfrac{\alpha\beta}{\lambda}-b_0)(\partial_t {\phi}_l, s_i) - \tfrac{\beta}{\lambda}(\partial_t {\xi}_k &, s_i) \nonumber \\
+ D ( \nabla\psi_m, \nabla s_i) + \gamma(\psi_m-\phi_l, s_i) &= (h, s_i) && \forall i=1,\dots,m,
\end{align} 
\end{subequations} 
with
\begin{align*}
    & (\bm{u}_j(0),\bm{v}_i) = \bm{u}_0 \quad \forall i=1,\dots,j, & \quad & (\xi_k(0),w_i) = \xi_0 \quad \forall i=1,\dots,k, \\
    & (\phi_l(0),q_i) = \phi_0 \quad \forall i=1,\dots,l, & \quad & (\psi_m(0),s_i) = \psi_0 \quad \forall i=1,\dots,m.
\end{align*}
We define the following block matrices corresponding to the bilinear forms:
\begin{align*}
& (\bm{A}_{uu})_{\iota\jmath} := 2\mu(\varepsilon(\bm{v}_\iota),\varepsilon(\bm{v}_\jmath)), \quad 
(\bm{B}_{u \xi})_{\iota\jmath} := (\nabla\cdot\bm{v}_\iota, w_\jmath), 
\\
& (\bm{C}_{\xi\xi})_{\iota\jmath} := \tfrac{1}{\lambda}(w_\iota, w_\jmath),  \quad (\bm{C}_{\xi\phi})_{\iota\jmath} := -\tfrac{\alpha}{\lambda}(w_\iota, q_\jmath), \quad
(\bm{C}_{\xi\psi})_{\iota\jmath} := -\tfrac{\beta}{\lambda}(w_\iota, s_\jmath), 
\\
& (\bm{\Theta}_{\phi\phi})_{\iota\jmath} := K(\nabla q_\iota, \nabla q_\jmath) + \gamma(q_\iota, q_\jmath),  \quad
(\bm{\Theta}_{\psi\psi})_{\iota\jmath} := D(\nabla s_\iota, \nabla s_\jmath) + \gamma(s_\iota, s_\jmath), 
\\
& (\bm{\Theta}_{\phi\psi})_{\iota\jmath} := -\gamma(q_\iota, s_\jmath), \quad
(\bm{M}_{\phi\phi})_{\iota\jmath} := (c_1 + \tfrac{\alpha^2}{\lambda})(q_\iota, q_\jmath), \quad
(\bm{M}_{\psi\psi})_{\iota\jmath} := (c_2 + \tfrac{\beta^2}{\lambda})(s_\iota,s_\jmath), \quad
\\
& (\bm{M}_{\phi\psi})_{\iota\jmath} := (\tfrac{\alpha\beta}{\lambda} - b_0)(q_\iota, s_\jmath), \quad
(\bm{M}_{\phi\xi})_{\iota\jmath} := -\tfrac{\alpha}{\lambda}(q_\iota,w_\jmath), \quad
(\bm{M}_{\psi\xi})_{\iota\jmath} := -\tfrac{\beta}{\lambda}(s_\iota,w_\jmath). \quad
\end{align*}
Let the load vectors be defined by:
\[(\bm{F})_\iota := (\bm{f}, \bm{v}_\iota),\quad (\bm{G})_\iota := (g, q_\iota),\quad (\bm{H})_\iota := (h, s_\iota). \]
We also introduce the coefficient vectors
$\bm U_j(t)$, $\bm\Xi_k(t)$, $\bm\Phi_l(t)$, and $\bm\Psi_m(t)$,
which denote the degrees of freedom of
$\bm u_j(t)$, $\xi_k(t)$, $\phi_l(t)$, and $\psi_m(t)$
with respect to the basis sets
$\{\bm v_i\}_{i=1}^j$, $\{w_i\}_{i=1}^k$, $\{q_i\}_{i=1}^l$, and $\{s_i\}_{i=1}^m$, respectively.  Using this notation, the problem \eqref{eq:discrete_weak_form} is equivalent to the following DAE system:
\begin{equation} \label{eq:DAE_system}
\mathbb{A} \tfrac{d}{dt} {\mathbf{X}}(t) + \mathbb{B}\mathbf{X}(t) = \mathbb{F}(t),
\end{equation}
where $\mathbf{X}(t) := [\mathbf{U}_j(t), \bm{\Xi}_k(t), \bm{\Phi}_l(t), \bm{\Psi}_m(t)]^T$, $\mathbb{F}(t) := [\bm{F}(t), \bm{0}, \bm{G}(t), \bm{H}(t)]^T$, and
\[
\mathbb{A} :=
\begin{pmatrix}
\bm{0} & \bm{0} & \bm{0} & \bm{0} \\
\bm{0} & \bm{0} & \bm{0} & \bm{0} \\
\bm{0} & \bm{M}_{\phi\xi} & \bm{M}_{\phi\phi} & \bm{M}_{\phi\psi} \\
\bm{0} & \bm{M}_{\psi\xi} & \bm{M}_{\phi\psi}^T & \bm{M}_{\psi\psi}
\end{pmatrix},
\quad
\mathbb{B} :=
\begin{pmatrix}
\bm{A}_{uu} & -\bm{B}_{u \xi} & \bm{0} & \bm{0} \\
\bm{B}_{u \xi}^T & \bm{C}_{\xi\xi} & \bm{C}_{\xi\phi} & \bm{C}_{\xi\psi} \\
\bm{0} & \bm{0} & \bm{\Theta}_{\phi\phi} & \bm{\Theta}_{\phi\psi} \\
\bm{0} & \bm{0} & \bm{\Theta}_{\phi\psi}^T & \bm{\Theta}_{\psi\psi}
\end{pmatrix}.
\]
To establish the well-posedness of the linear DAE system \eqref{eq:DAE_system}, it suffices to demonstrate that the linear combination $s\mathbb{A} + \mathbb{B}$ is non-singular for some $s \not= 0$. The following lemma verifies this condition for the case $s=1$. 

To simplify the subsequent analysis, we introduce two functionals $\mathcal{E}$ and $\mathcal{V}$ acting on generic functions $(\bm v, w, q, s)$:
\begin{align}
\mathcal{E}(\bm v, w, q, s) &:= 2 \mu\|\varepsilon(\bm v)\|^2 + \tfrac{1}{\lambda}\| w - \alpha q - \beta s \|^2 
\nonumber \\ 
& \quad \  + (c_1-b_0) \|q\|^2
+ (c_2-b_0)\|s\|^2 + b_0 \| q - s\|^2, \label{def:E_func} \\
\mathcal{V}(q, s) & := K\|\nabla q\|^2 + D\|\nabla s\|^2 + \gamma\|q-s\|^2. \label{def:V_func}
\end{align}

\begin{lemma}
\label{thm1s}
The system matrix $\mathbb{A} + \mathbb{B}$ is positive definite and, therefore, invertible. For any $(j, k, l, m) \in \mathbb{N}^4$, the linear DAE system \eqref{eq:DAE_system} admits a unique solution.
\end{lemma}

\begin{proof}
Let $\bm{Z} := [\bm{v}, w, q, s]^T$ denote an arbitrary vector representing the degrees of freedom in the finite-dimensional space $\bm{V}_j \times W_k \times Q_l \times S_m$. We consider the quadratic form associated with the matrix $\mathbb{M}= \mathbb{A} + \mathbb{B}$, defined by $\mathcal{M}(\bm{Z}) := \bm{Z}^T \mathbb{M} \bm{Z}$. To prove invertibility, it suffices to show that $\mathbb{M}$ is strictly positive definite, i.e., $\mathcal{M}(\bm{Z}) > 0$ for all $\bm{Z} \neq \bm{0}$.

Substituting the block definitions of $\mathbb{A}$ and $\mathbb{B}$, and observing that the terms involving the divergence operator cancel out, we obtain the expansion:
\begin{align}
\mathcal{M}(\bm{Z}) = & \  2\mu\|\varepsilon(\bm{v})\|^2 + \tfrac{1}{\lambda}\|w\|^2 + (c_1+\tfrac{\alpha^2}{\lambda})\|q\|^2 + (c_2+\tfrac{\beta^2}{\lambda})\|s\|^2
\nonumber \\ & - \tfrac{2}{\lambda}(w, \alpha q+\beta s) + 2(\tfrac{\alpha\beta}{\lambda}-b_0)(q, s) + K\|\nabla q\|^2 + D\|\nabla s\|^2 + \gamma\|q-s\|^2.
\label{eq:expansion_raw}
\end{align}
To establish coercivity, we employ the following algebraic identities to regroup as follows:
\begin{align*}
    \tfrac{1}{\lambda} \| w - \alpha q - \beta s\|^2 & =  \tfrac{1}{\lambda} \|\xi\|^2 + \tfrac{\alpha^2}{\lambda}\|q\|^2 + \tfrac{\beta^2}{\lambda}\|s\|^2 - \tfrac{2}{\lambda}(w, \alpha q+\beta s) + \tfrac{2\alpha\beta}{\lambda}(q, s),  
    \\
    b_0\|q - s\|^2 & = b_0 \|q\|^2 + b_0\|s\|^2 - 2b_0 (q , s). 
\end{align*}
Substituting these identities back into \eqref{eq:expansion_raw}, the expression simplifies to: 
\begin{align} 
\mathcal{M}(\bm{Z}) = \mathcal{E}(\bm v, w, q, s) + \mathcal{V}(q, s). \label{eq:matrix_coercivity}
\end{align}
By Assumption \ref{assumption1}, all physical parameters are non-negative, the coefficients satisfy $c_1 \ge b_0$ and $c_2 \ge b_0$. Thus, $\mathcal{M}(\bm{Z}) \ge 0$. Now, assume $\mathcal{M}(\bm{Z}) = 0$ in \eqref{eq:matrix_coercivity}. Since $\mathcal{E}$ and $\mathcal{V}$ are sums of non-negative terms, each must vanish individually:
From $\mathcal{V}(q,s) = 0$, we deduce $K\|\nabla q\|^2 = 0$ and $D\|\nabla s\|^2 = 0$. Applying the Poincaré inequality yields $q = 0$ and $s = 0$.
From $\mathcal{E}(\bm v,w,q,s) = 0$, we deduce $2\mu\|\varepsilon(\bm v)\|^2 = 0$. Applying Korn's inequality yields $\bm v = \bm 0$. Furthermore, checking the term $\tfrac{1}{\lambda} \| w - \alpha q - \beta s\|^2 = 0$ with $q = s = 0$ gives $w = 0$.
Therefore, $\mathcal{M}(\bm{Z}) = 0$ implies $\bm{Z} = \bm{0}$. The matrix $\mathbb{M} = \mathbb{A} + \mathbb{B}$ is strictly positive definite and invertible.
\end{proof}

With the well-posedness of the discrete system established, we now aim to derive uniform a priori estimates for the solution. These bounds are independent of the discretization parameters $j, k, l,$ and $m$, which is essential for passing to the limit. For the sake of brevity in the following analysis, we drop the discretization subscripts and denote the discrete solution $(\bm{u}_j, \xi_k, \phi_l, \psi_m)$ simply as $(\bm{u}, \xi, \phi, \psi)$. In the following estimates, we use the notation $X \lesssim Y$ to denote $X \leq C Y$, where $C > 0$ is a constant independent of the indices $(j, k, l, m)$.

\begin{lemma}
\label{lem:priori}
Let $(\bm{u}, \xi, \phi, \psi)$ be the solution to the discrete system \eqref{eq:discrete_weak_form}. The following stability estimates hold:
\begin{align}
     & \mu \| \bm u \|_{L^\infty(0,T_f;\bm{V})}^2 + K \| \phi \|_{L^2(0,T_f;Q)}^2 + D \| \psi \|_{L^2(0,T_f;S)}^2 
     \nonumber \\
     & \lesssim 
     \mathcal{E}(\bm u_0, \xi_0, \phi_0, \psi_0) +  \|\bm f\|_{H^1(0,T_f;[L^2(\Omega)]^d)}^2 + \|g\|_{L^2(0,T_f;L^2(\Omega))}^2 + \|h\|_{L^2(0,T_f;L^2(\Omega))}^2 , \label{est:stability} 
     \\
     & \mu \| \bm u \|_{H^1(0,T_f;\bm V)}^2 + K \| \phi \|_{L^\infty(0,T_f;Q)}^2 + D \| \psi \|_{L^\infty(0,T_f;S)}^2 \nonumber \\
     & \lesssim
    \mathcal{V}(\phi_0, \psi_0) + \|\bm f\|_{H^1(0,T_f;[L^2(\Omega)]^d)}^2 + \|g\|_{H^1(0,T_f;L^2(\Omega))}^2 + \|h\|_{H^1(0,T_f;L^2(\Omega))}^2 . \label{est:stability2} 
\end{align}
Furthermore, the variable $\xi$ satisfies:
\begin{align}
C_{is} \| \xi \|_{L^\infty(0,T_f;W)}  & \lesssim \mu \| \bm  u \|_{L^\infty(0,T_f;\bm{V})} + C_p \|\bm f\|_{L^\infty(0,T_f;L^2(\Omega))}, 
\label{est:stabilityA}
\\
C_{is} \| \xi \|_{H^1(0,T_f;W)}  & \lesssim \mu \|  \bm  u \|_{H^1(0,T_f;\bm V)} + C_p \| \bm f\|_{H^1(0,T_f;L^2(\Omega))}.
\label{est:stability2A}
\end{align}
\end{lemma}

\begin{proof}
Differentiating equation \eqref{wf:constraint} with respect to time yields:
\begin{align}
    (\nabla\cdot\partial_t {\bm u},w) + \tfrac{1}{\lambda}(\partial_t \xi,w) - \tfrac{\alpha}{\lambda}(\partial_t \phi,w) - \tfrac{\beta}{\lambda}(\partial_t \psi,w) & = 0. \label{dtw2}
\end{align}
We verify the identity by choosing the test functions $\bm{v} = \partial_t {\bm{u}}$ in \eqref{wf:momentum}, $w = \xi$ in \eqref{dtw2}, $q = \phi$ in \eqref{wf:phi}, and $s = \psi$ in \eqref{wf:psi}. Summing the resulting equations and eliminating the coupling terms involving the divergence operator leads to:
\begin{align}
    & 2 \mu (\varepsilon(\bm u), \varepsilon( \partial_t\bm u) ) + \tfrac{1}{\lambda} (\partial_t \xi , \xi ) + (c_1+\tfrac{\alpha^2}{\lambda} ) (\partial_t \phi, \phi) + (c_2+\tfrac{\beta^2}{\lambda} ) (\partial_t \psi, \psi ) + (\tfrac{\alpha\beta}{\lambda} - b_0)(\partial_t \psi, \phi) \nonumber \\
    & + (\tfrac{\alpha\beta}{\lambda} - b_0)(\partial_t \phi, \psi) - \tfrac{\alpha}{\lambda}(\partial_t \phi, \xi) - \tfrac{\beta}{\lambda}(\partial_t \psi, \xi) - \tfrac{\alpha}{\lambda}(\partial_t \xi, \phi) - \tfrac{\beta}{\lambda}(\partial_t \xi, \psi)  \nonumber \\ 
    & + K\|\nabla\phi\|^2 + D\|\nabla\psi\|^2 + \gamma\|\phi-\psi\|^2 
     = (\bm f,\partial_t {\bm u}) + (g,\phi) + (h,\psi). \label{lem1:eq1}
\end{align}
We observe that the cross terms in \eqref{lem1:eq1} can be regrouped as follows:
\begin{align}
    & \tfrac{1}{\lambda}(\partial_t\xi - \alpha\partial_t\phi - \beta\partial_t\psi, \xi - \alpha \phi - \beta \psi) =  \tfrac{1}{\lambda} (\partial_t \xi , \xi ) + \tfrac{\alpha^2}{\lambda}(\partial_t \phi, \phi) + \tfrac{\beta^2}{\lambda}(\partial_t \psi, \psi ) \nonumber \\
    & \quad + \tfrac{\alpha \beta}{\lambda} (\partial_t \psi, \phi) + \tfrac{\alpha \beta}{\lambda}  (\partial_t \phi, \psi) - \tfrac{\alpha }{\lambda} (\partial_t \phi, \xi) - \tfrac{\beta}{\lambda} (\partial_t \psi, \xi) - \tfrac{\alpha}{\lambda} (\partial_t \xi, \phi) - \tfrac{ \beta}{\lambda} (\partial_t \xi, \psi),  \label{lem1:eq2} \\
    & b_0 (\partial_t \phi - \partial_t \psi, \phi - \psi) = b_0 (\partial_t \phi , \phi) + b_0(\partial_t \psi , \psi) - b_0 (\partial_t \phi , \psi) - b_0 (\partial_t \psi , \phi). \label{lem1:eq3}
\end{align}
Applying the identities \eqref{lem1:eq2} and \eqref{lem1:eq3} to \eqref{lem1:eq1}, along with $( \partial_t {x}, x ) = \frac{1}{2}\frac{d}{dt}\|x\|^2$, we arrive at:
\begin{align}
    & \tfrac{1}{2} \tfrac{d}{dt} \mathcal{E}(\bm u, \xi, \phi, \psi)
    + \mathcal{V}( \phi, \psi)
     = (\bm f,\partial_t {\bm u}) + (g,\phi) + (h,\psi).
\end{align}
Integrating this differential equality over $(0, \tau)$ yields the energy identity: 
\begin{align} 
& \tfrac{1}{2} \mathcal{E}(\bm u(\tau), \xi(\tau), \phi(\tau), \psi(\tau)) + \int_0^\tau \mathcal{V}( \phi, \psi) \, ds 
\nonumber \\ 
& = \tfrac{1}{2} \mathcal{E}(\bm u_0, \xi_0, \phi_0, \psi_0) + \int_0^\tau \big[ (\bm f, \partial_t {\bm u}) + (g, \phi) + (h, \psi) \big ] \, ds. 
\label{eq:energy_identity} 
\end{align}

We start by estimating the terms on the right-hand side of the identity \eqref{eq:energy_identity}. Applying integration by parts in time to the term $\int_0^\tau (\bm f, \partial_t \bm u) ds$, followed by the Cauchy-Schwarz, Young's, and Korn's inequalities, we obtain:
\begin{align} 
    & \int_0^\tau (\bm f, \partial_t \bm u) \, ds  = (\bm f(\tau), \bm u(\tau)) - (\bm f(0), \bm u(0)) -  \int_0^\tau (\partial_t \bm f, \bm u) \, ds  \nonumber \\ 
    & \leq C \| \bm f \|_{H^1(0,T_f;[L^2(\Omega)]^d)}^2 + \tfrac{\mu}{2} \| \varepsilon(\bm u(\tau)) \|^2 + \tfrac{\mu}{2}\|\varepsilon(\bm u_0)\|^2 + \int_0^\tau \tfrac{\mu}{2} \| \varepsilon (\bm u) \|^2 \, ds. \label{thm2:ieq1}
\end{align}
Similarly, applying the Cauchy-Schwarz, Young's, and Poincaré inequalities yields:
\begin{align*}
    \int_0^\tau (g,\phi) \, ds \leq C \|g\|_{L^2(0,T_f;L^2(\Omega))}^2 + \frac{K}{2}\int_0^\tau \| \nabla \phi \|^2 \, ds, \\
    \int_0^\tau (h,\psi) \, ds \leq C \|h\|_{L^2(0,T_f;L^2(\Omega))}^2 + \frac{D}{2}\int_0^\tau \| \nabla \psi \|^2 \, ds.
\end{align*}
Absorbing the terms $\tfrac{\mu}{2}\|\varepsilon(\bm u(\tau))\|^2$, $\tfrac{K}{2}\int_0^\tau\|\nabla \phi\|^2\, ds$, and $\tfrac{D}{2}\int_0^\tau\|\nabla \psi\|^2\, ds$ into the left-hand side of \eqref{eq:energy_identity}, incorporating the $\tfrac{\mu}{2}\|\varepsilon(\bm u_0)\|^2$ term into the initial energy, and applying Gronwall’s inequality to handle the integral term $\int_0^\tau \tfrac{\mu}{2} \| \varepsilon (\bm u) \|^2 \, ds$, we arrive at:
\begin{align}
     & \mathcal{E}(\bm u(\tau), \xi(\tau), \phi(\tau), \psi(\tau)) + K \| \phi \|_{L^2(0,T_f;Q)}^2 + D \| \psi \|_{L^2(0,T_f;S)}^2 
     \nonumber \\
     & \lesssim 
     \mathcal{E}(\bm u_0, \xi_0, \phi_0, \psi_0) + \|\bm f\|_{H^1(0,T_f;[L^2(\Omega)]^d)}^2 + \|g\|_{L^2(0,T_f;L^2(\Omega))}^2 + \|h\|_{L^2(0,T_f;L^2(\Omega))}^2. \label{est:stabilityB} 
\end{align}
Finally, applying the inf-sup condition \eqref{infsupcon} gives:
\begin{align}
C_{is} \| \xi(\tau) \| & \leq \sup_{\bm 0 \not = \bm{v}\in \bm{V}} \frac{(\xi, \nabla \cdot \bm{v})}{\| \bm{v} \|_{H^1}} 
= \sup_{\bm 0 \not = \bm{v}\in \bm{V}} \frac{2\mu (\varepsilon(\bm u), \varepsilon(\bm{v})) - (\bm f, \bm v)}{\| \bm{v} \|_{H^1}} \nonumber \\
& \leq 2\mu \| \varepsilon(\bm  u(\tau)) \| + C_p \|\bm f(\tau)\|. \label{eqf:mech_infsup}
\end{align}
Combining \eqref{est:stabilityB} and \eqref{eqf:mech_infsup} concludes \eqref{est:stability} and \eqref{est:stabilityA}.

Next, we differentiate the momentum equation \eqref{wf:momentum} with respect to time:
\begin{align}
    2\mu(\varepsilon(\partial_t\bm u),\varepsilon(\bm v)) - (\partial_t\xi,\nabla\cdot\bm v) =  (\partial_t\bm f,\bm v). \label{dtw1}
\end{align}
We then test \eqref{dtw1} with $\bm v= \partial_t\bm u$, \eqref{dtw2} with $w = \partial_t \xi$, \eqref{wf:phi} with $q = \partial_t \phi$, and \eqref{wf:psi} with $s = \partial_t \psi$. Recalling the algebraic structure established in \eqref{lem1:eq2} and invoking the chain rule for the gradient terms, summing these contributions leads to:
\begin{align*}
& \mathcal{E}(\partial_t \bm u, \partial_t \xi, \partial_t \phi, \partial_t \psi) + \tfrac{1}{2} \tfrac{d}{dt} \mathcal{V}( \phi, \psi)
= (\partial_t \bm f,\partial_t \bm u)
+ (g,\partial_t \phi)
+ (h,\partial_t \psi).
\end{align*}
Integrating over $(0, \tau)$ yields
\begin{align} 
& \int_0^\tau \mathcal{E}(\partial_t \bm u, \partial_t \xi, \partial_t \phi, \partial_t \psi) \, ds + \tfrac{1}{2} \mathcal{V}( \phi(\tau), \psi(\tau)) \nonumber \\ 
& = \tfrac{1}{2} \mathcal{V}( \phi_0, \psi_0) + \int_0^\tau \big[ (\partial_t \bm f, \partial_t \bm u) + (g, \partial_t \phi) + (h, \partial_t \psi) \big ] \, ds. \label{eq:high_order_identity} 
\end{align}

We next address the right-hand side of the \eqref{eq:high_order_identity}. Estimating the source terms using integration by parts in time, analogous to \eqref{thm2:ieq1}, yields
\begin{align*} 
    & \int_0^\tau (\partial_t \bm f , \partial_t \bm u) \, ds \leq C\| \bm f \|_{H^1(0,T_f;[L^2(\Omega)]^d)}^2 + \int_0^\tau \mu \| \varepsilon (\partial_t \bm u) \|^2 \, ds, \\
    & \int_0^\tau (g , \partial_t \phi) \, ds \leq C \| g \|_{H^1(0,T_f;L^2(\Omega))}^2 + \tfrac{K}{4} \| \nabla \phi(\tau) \|^2 + \tfrac{K}{4} \| \nabla \phi_0\|^2 + \int_0^\tau \tfrac{K}{4}  \| \nabla \phi \|^2 \, ds,  \\
    & \int_0^\tau (h , \partial_t \psi) \, ds \leq C \| h \|_{H^1(0,T_f;L^2(\Omega))}^2 + \tfrac{D}{4} \| \nabla \psi(\tau) \|^2 + \tfrac{D}{4} \| \nabla \psi_0\|^2 + \int_0^\tau \tfrac{D}{4}  \| \nabla \psi \|^2 \, ds. 
\end{align*}
Absorbing the terms $\int_0^\tau \mu \| \varepsilon (\partial_t \bm u) \|^2 \, ds$, $\tfrac{K}{4}\|\nabla \phi(\tau)\|^2$,  and $\tfrac{D}{4}\|\nabla \psi(\tau)\|^2$ into the left-hand side of \eqref{eq:high_order_identity}, incorporating the $\tfrac{K}{4}\|\nabla \phi_0\|^2$ and $\tfrac{D}{4}\|\nabla \psi_0\|^2$ terms into the initial energy, and applying Gronwall’s
inequality to handle the terms $\int_0^\tau \tfrac{K}{4}  \| \nabla \phi \|^2 \, ds$ and $\int_0^\tau \tfrac{D}{4}  \| \nabla \psi \|^2 \, ds$, we arrive at:
\begin{align}
& \int_0^\tau \mu \| \varepsilon(\partial_t \bm u) \|^2 \, ds + \mathcal{V}(\phi(\tau), \psi(\tau)) \nonumber \\
& \lesssim \mathcal{V}(\phi_0, \psi_0)  + \|\bm f\|_{H^1(0,T_f;[L^2(\Omega)]^d)}^2 + \|g\|_{H^1(0,T_f;L^2(\Omega))}^2 + \|h\|_{H^1(0,T_f;L^2(\Omega))}^2. \label{est:stability2_inter}
\end{align}
Finally, applying the same technique used in \eqref{eqf:mech_infsup} to \eqref{dtw1} yields:
\begin{align}
C_{is} \| \partial_t \xi \|  \leq 2\mu \| \varepsilon( \partial_t \bm  u) \| + C_p \|\partial_t \bm f\|. \label{eqf:mech_infsup2}
\end{align}
Combining \eqref{est:stability2_inter} and \eqref{eqf:mech_infsup2} concludes \eqref{est:stability2} and \eqref{est:stability2A}.
\end{proof}

With the uniform a priori estimates of Lemma~\ref{lem:priori}, we can extract weakly convergent subsequences. Using compactness arguments and passing to the limit $j,k,l,m \to \infty$, we obtain the existence of a weak solution
to \eqref{weakform}. Uniqueness follows by considering the difference of two solutions, which satisfies the homogeneous system with zero initial data.
Applying the energy identities \eqref{eq:energy_identity} and \eqref{eq:high_order_identity} implies that the difference vanishes identically.

\subsection{Space-time iterative approximation}

We introduce a space-time iterative decoupled scheme at the continuous level,
which serves as a fixed-point approximation of the coupled system
\eqref{weakform}. Let the superscript $i \geq 1$ denote the current iteration index. For the initial guess $i = 0$, one may choose an arbitrary state, though a common choice is to use the initial conditions. Given the solution from the previous iteration, specifically the total pressure $\xi^{i-1}$, the coupled problem is split into two sequential sub-problems:

\paragraph{Step 1: Transport sub-problem}
First, we update the generalized pressure variables by treating the term involving $\partial_t \xi$ explicitly using the solution from the previous iteration. For almost every $t \in (0,T_f]$, find $(\phi^i, \psi^i) \in Q\times S$ such that
\begin{subequations}
\label{iter:step1}
\begin{align}
\big(c_1+\tfrac{\alpha^2}{\lambda}\big)(\partial_t {\phi}^i, q) + \big(\tfrac{\alpha\beta}{\lambda}-b_0\big)(\partial_t {\psi}^i, q)  + K( \nabla\phi^i,\nabla q) \qquad \qquad \qquad \nonumber \\
\ + \gamma(\phi^i-\psi^i, q) 
= (g,q) + \tfrac{\alpha}{\lambda}(\partial_t {\xi}^{i-1}, q) \qquad \forall q \in Q,\\
\big(c_2+\tfrac{\beta^2}{\lambda}\big)(\partial_t {\psi}^i, s) + \big(\tfrac{\alpha\beta}{\lambda}-b_0\big) (\partial_t {\phi}^i, s)  + D( \nabla\psi^i,\nabla s) \qquad \qquad \qquad \nonumber \\
\ + \gamma(\psi^i-\phi^i, s) 
= (h,s) + \tfrac{\beta}{\lambda}(\partial_t {\xi}^{i-1}, s) \qquad \forall s \in S.
\end{align}
\end{subequations}

\paragraph{Step 2: Mechanics sub-problem}
Next, we update the mechanical variables using the transport variables computed in \textit{Step 1}. For almost every $t \in (0,T_f]$, find $(\bm u^i , \xi^i) \in \bm{V} \times W$ such that
\begin{subequations}
\label{iter:step2}
\begin{align}
& 2\mu(\varepsilon(\bm u^i),\varepsilon(\bm v)) - (\xi^i,\nabla\cdot\bm v)  = (\bm f,\bm v) \qquad & \forall & \bm v \in \bm V, \\
& (\nabla\cdot\bm u^i, w) +\tfrac{1}{\lambda} (\xi^i, w) = \tfrac{\alpha}{\lambda} (\phi^i, w) + \tfrac{\beta}{\lambda} (\psi^i, w) \qquad & \forall & w \in W.
\end{align}
\end{subequations}

The above two-step procedure defines a continuous fixed-point iteration
for the coupled system \eqref{weakform}. To analyze the convergence of
this iterative approximation, we introduce the errors between the approximate sequence $\{(\bm u^i , \xi^i,\phi^i, \psi^i)\}_{i\geq 0}$ generated by \eqref{iter:step1}--\eqref{iter:step2} and the exact solution $(\bm u, \xi,\phi, \psi)$ to the system \eqref{weakform} as:
\begin{align*}
e_{\bm{u}}^i = \bm{u}^i - \bm{u}, \quad e_\xi^i = \xi^i - \xi, \quad
e_\phi^i = \phi^i - \phi, \quad e_\psi^i = \psi^i - \psi.
\end{align*}
The following theorem establishes the convergence of the iterative
procedure by showing that these errors vanish as $i \to \infty$.


\begin{theorem}
\label{thm:convergence}
The iterative decoupled scheme \eqref{iter:step1}--\eqref{iter:step2} is strictly contractive. Specifically, the following error estimate holds:
\begin{equation}
    \|\partial_t {e}_{\xi}^{i}\|_{L^2(0,T_f; W)} \leq L_{\mathrm{con}} \|\partial_t {e}_{\xi}^{i-1}\|_{L^2(0,T_f; W)}, \label{thm:contraction}
\end{equation}
where the contraction factor is given by
\begin{equation}
    L_{\mathrm{con}} := \frac{1}{\sqrt{1 + 2\lambda C_\mu}} < 1,  \label{def:Lcon}
\end{equation}
with $C_\mu$ being a positive constant. As $i \to \infty$, the sequence of iterates converges to the unique weak solution of the coupled system in the following sense:
\begin{align*}
    & \bm{e}_u^i \to \bm{0} \quad \text{in } L^\infty(0,T_f; \bm{V}), & 
    & e_\xi^i \to 0 \quad \text{in } L^\infty(0,T_f; W), 
    \\
    & e_\phi^i \to 0 \quad \text{in } L^\infty(0,T_f; Q), & 
    & e_\psi^i \to 0 \quad \text{in } L^\infty(0,T_f; S).
\end{align*}
Consequently, the iterative scheme is globally convergent.
\end{theorem}

\begin{proof}
The proof is structured in three steps. We first derive error estimates for the transport and mechanics sub-problems separately, and subsequently combine them to demonstrate the contraction of the iterative scheme.

\noindent
\textbf{1. Transport error equations.}
Subtracting the exact equations \eqref{wf:momentum}--\eqref{wf:constraint} from the transport sub-problem \eqref{iter:step1} yields the error equations:
\begin{align}
(c_1+\tfrac{\alpha^2}{\lambda})(\partial_t  e_{\phi}^i, q) + (\tfrac{\alpha\beta}{\lambda}-b_0)(\partial_t  e_{\psi}^i, q)  + K(\nabla e_\phi^i,\nabla q) + \gamma( e_\phi^i- e_\psi^i, q) & = \tfrac{\alpha}{\lambda}(\partial_t  e_{\xi}^{i-1}, q), \label{con:eq1}
\\
(c_2+\tfrac{\beta^2}{\lambda})(\partial_t  e_{\psi}^i, s) + (\tfrac{\alpha\beta}{\lambda}-b_0) (\partial_t  e_{\phi}^i, s)  + D(\nabla e_\psi^i,\nabla s) + \gamma(e_\psi^i-e_\phi^i, s) & = \tfrac{\beta}{\lambda}(\partial_t  e_{\xi}^{i-1}, s). \label{con:eq2}
\end{align}
Choosing the test functions $q = \partial_t {e}_\phi^i$ in \eqref{con:eq1}, $s = \partial_t {e}_\psi^i$ in \eqref{con:eq2}, we sum the results to obtain:
\begin{align}
& (c_1-b_0) \| \partial_t  e_{\phi}^i \|^2 + (c_2-b_0) \| \partial_t  e_{\psi}^i \|^2 + b_0 \| \partial_t  e_{\phi}^i - \partial_t {e}_\psi^i \|^2  + \tfrac{1}{\lambda} \| \alpha \partial_t  e_{\phi}^i +  \beta \partial_t {e}_\psi^i \|^2
\nonumber \\
& + \tfrac{1}{2}\tfrac{d}{dt} ( K \| \nabla e_\phi^i \|^2 + D \| \nabla e_\psi^i \|^2 + \gamma \| e_\phi^i- e_\psi^i \|^2)
= \tfrac{1}{\lambda}(\partial_t {e}_{\xi}^{i-1}, \alpha \partial_t  e_{\phi}^i + \beta \partial_t {e}_\psi^i).
\label{con:eq3}
\end{align}
Applying the Cauchy--Schwarz and Young inequalities to the right-hand side:
\begin{equation}
\tfrac{1}{\lambda}(\partial_t {e}_{\xi}^{i-1}, \alpha\partial_t {e}_\phi^i + \beta\partial_t {e}_\psi^i) \le \tfrac{1}{2\lambda}  \|\partial_t {e}_{\xi}^{i-1}\|^2 + \tfrac{1}{2\lambda} \|\alpha\partial_t {e}_\phi^i + \beta\partial_t {e}_\psi^i\|^2.
\label{con:eq4}
\end{equation}
Combining \eqref{con:eq3} and \eqref{con:eq4}, integrating over $(0, T_f)$, and using the fact that initial errors at $t=0$ are zero (i.e., $e_\phi^i(0) = e_\psi^i(0) = 0$), we obtain
\begin{align}
& 2(c_1-b_0) \int_0^{T_f} \| \partial_t  e_{\phi}^i \|^2 ds + 2(c_2-b_0) \int_0^{T_f} \| \partial_t  e_{\psi}^i \|^2 ds + 2b_0 \int_0^{T_f} \| \partial_t  e_{\phi}^i - \partial_t {e}_\psi^i \|^2 ds 
\nonumber \\
& + K \| \nabla e_\phi^i (T_f) \|^2 + D \| \nabla e_\psi^i (T_f) \|^2 + \gamma \| e_\phi^i (T_f) - e_\psi^i (T_f) \|^2
\nonumber \\
& + \tfrac{1}{\lambda} \int_0^{T_f} \| \alpha \partial_t  e_{\phi}^i +  \beta \partial_t {e}_\psi^i \|^2 ds \leq \tfrac{1}{\lambda} \int_0^{T_f} \|\partial_t {e}_{\xi}^{i-1}\|^2 ds, \label{est:transport_finalp}
\end{align}
Finally, disregarding the non-negative terms on the left-hand side of \eqref{est:transport_finalp}, we arrive at the following estimate for the transport step:
\begin{align}
\|\alpha\partial_t {e}_\phi^i + \beta\partial_t {e}_\psi^i\|_{L^2(0,T_f; L^2(\Omega))}^2   \le   \|\partial_t {e}_{\xi}^{i-1}\|_{L^2(0,T_f; W)}^2  . \label{est:transport_final}
\end{align}

\noindent
\textbf{2. Mechanics error equations.}
Subtracting the exact equations \eqref{wf:phi}--\eqref{wf:psi} from the mechanics sub-problem \eqref{iter:step2} yields the error equations:
\begin{align}
& 2\mu(\varepsilon( \bm e_u^i),\varepsilon(  \bm v)) - (  e_\xi^i,\nabla\cdot \bm v) = 0, \label{err:mom} \\
& (\nabla\cdot \bm e_u^i, w) +\tfrac{1}{\lambda} ( e_\xi^i, w) = \tfrac{\alpha}{\lambda} ( e_\phi^i, w) + \tfrac{\beta}{\lambda} ( e_\psi^i, w). \label{err:mass}
\end{align}
Differentiating \eqref{err:mom} and \eqref{err:mass} with respect to time yields
\begin{align}
& 2\mu(\varepsilon( \partial_t \bm e_u^i),\varepsilon(  \bm v)) - (  \partial_t e_\xi^i,\nabla\cdot \bm v) = 0, \label{err:dtmom} \\
& (\nabla\cdot \partial_t \bm e_u^i, w) +\tfrac{1}{\lambda} ( \partial_t e_\xi^i, w) = \tfrac{\alpha}{\lambda} ( \partial_t e_\phi^i, w) + \tfrac{\beta}{\lambda} ( \partial_t e_\psi^i, w). \label{err:dtmass}
\end{align}
Choosing the test functions $\bm{v} = \partial_t {\bm e}_u^i$ in \eqref{err:dtmom} and $w = \partial_t {e}_\xi^i$ in \eqref{err:dtmass}, we sum the results to get:
\begin{align}
2\mu \| \varepsilon( \partial_t {\bm e}_u^i) \|^2 + \tfrac{1}{\lambda} \| \partial_t  e_\xi^i \|^2  = \tfrac{1}{\lambda} (\alpha \partial_t   e_\phi^i + \beta \partial_t  e_\psi^i, \partial_t {e}_\xi^i). \label{eq:mecha_identity}
\end{align} 
Applying the Cauchy--Schwarz and Young inequalities to the right-hand side:
\begin{equation}
\tfrac{1}{\lambda}(\partial_t {e}_{\xi}^{i}, \alpha\partial_t {e}_\phi^i + \beta\partial_t {e}_\psi^i) \le \tfrac{1}{2\lambda}  \|\partial_t {e}_{\xi}^{i}\|^2 + \tfrac{1}{2\lambda} \|\alpha\partial_t {e}_\phi^i + \beta\partial_t {e}_\psi^i\|^2.
\label{con:eq6}
\end{equation}
Applying the inf-sup condition \eqref{infsupcon} to \eqref{err:dtmom}, we obtain
\begin{align*}
C_{is} \| \partial_t  e_\xi^i \| &\leq \sup_{\bm 0 \not = \bm{v}\in \bm{V}} \frac{(\partial_t  e_\xi^i, \nabla \cdot \bm{v})}{\| \bm{v} \|_{H^1}} 
= \sup_{\bm 0 \not = \bm{v}\in \bm{V}} \frac{2\mu (\varepsilon(\partial_t {\bm e}_u^i), \varepsilon(\bm{v}))}{\| \bm{v} \|_{H^1}} 
\leq 2\mu \| \varepsilon(\partial_t {\bm e}_u^i) \|, 
\end{align*}
implying that there exists \( C_\mu>0\) (independent of \(h\) and \(\Delta t\)) such that
\begin{align}
C_\mu \|\partial_t e_\xi^{i}\|^2 \le 2\mu\|\varepsilon(\partial_t \bm e_u^{i})\|^2. \label{xibound}
\end{align}
Substituting the bounds \eqref{con:eq6} and \eqref{xibound} into \eqref{eq:mecha_identity}, integrating over $(0,T_f)$, we obtain:
\begin{align}
& (C_\mu+\tfrac{1}{2\lambda})\|\partial_t  e_\xi^i\|_{L^2(0,T_f; W)}^2 = \int_0^T \left( C_\mu \| \partial_t  e_\xi^i \|^2 + \tfrac{1}{2\lambda} \| \partial_t  e_\xi^i \|^2 \right) ds 
\nonumber \\
& \qquad \leq \int_0^T \tfrac{1}{2\lambda} \|\alpha \partial_t   e_\phi^i + \beta \partial_t  e_\psi^i\|^2 ds =  \tfrac{1}{2\lambda} \| \alpha \partial_t   e_\phi^i + \beta \partial_t  e_\psi^i \|_{L^2(0,T_f; L^2(\Omega))}^2.
\label{est:mechanics}
\end{align}

\noindent
\textbf{3. Contraction and convergence.}
Combining the estimates \eqref{est:transport_final} and \eqref{est:mechanics}, we obtain the contraction inequality:
\begin{align*}
(C_\mu+\tfrac{1}{2\lambda})\|\partial_t  e_\xi^i\|_{L^2(0,T_f; W)}^2 \le \tfrac{1}{2\lambda} \| \alpha \partial_t   e_\phi^i + \beta \partial_t  e_\psi^i \|_{L^2(0,T_f; L^2(\Omega))}^2 \le \tfrac{1}{2\lambda} \|\partial_t  e_\xi^{i-1}\|_{L^2(0,T_f; W)}^2.
\end{align*}
This proves the contraction property \eqref{thm:contraction}. Since $C_\mu > 0$ and $\lambda > 0$, we have $L_{\mathrm{con}}$ defined in \eqref{def:Lcon} strictly smaller than $1$. To recover the convergence of the variable itself from its time derivative, we recall that the initial error is zero (i.e., $e_\xi^i(0) = 0$). Applying the Fundamental Theorem of Calculus and the Cauchy--Schwarz inequality, we obtain the bound:
\begin{align}
\| e_\xi^i(\tau) \|^2 = \left\| \int_0^\tau \partial_t {e}_\xi^i(s) \, ds \right\|^2 \le \tau \int_0^\tau \| \partial_t {e}_\xi^i(s) \|^2 \, ds \le T_f \| \partial_t {e}_\xi^i \|_{L^2(0,T_f; W)}^2, \label{thm1:con1}
\end{align}
for any $\tau \in (0, T_f]$.
This implies that $e_\xi^i$ converges to zero in $L^\infty(0, T_f; W)$ norm as $i \to \infty$.

We now extend this convergence result to the other variables. For the displacement, testing the equation \eqref{err:dtmom} with $\bm{v} = \partial_t \bm{e}_u^i$ yields $2\mu \|\varepsilon(\partial_t \bm{e}_u^i)\|^2 = (\partial_t e_\xi^i, \nabla \cdot \partial_t \bm{e}_u^i)$. Using the Cauchy--Schwarz and Korn's inequalities, we deduce that $\|\partial_t \bm{e}_u^i\|_{H^1} \lesssim \|\partial_t e_\xi^i\|$. Applying the same argument as in \eqref{thm1:con1}, we conclude that $\bm{e}_u^i \to \bm{0}$ in $L^\infty(0,T_f; \bm{V})$. 
For the generalized pressures, we refer to the transport energy estimate \eqref{est:transport_finalp}. As the right-hand side is bounded by $\|\partial_t e_\xi^{i-1}\|^2$, the non-negativity of the terms on the left-hand side implies that $K\|\nabla e_\phi^i(T_f)\|^2 \to 0$ and $D\|\nabla e_\psi^i(T_f)\|^2 \to 0$. Since this estimate holds for any arbitrary final time $T_f$, it follows that $e_\phi^i \to 0 \ \text{in } L^\infty(0,T_f; Q)$, and $e_\psi^i \to 0 \ \text{in } L^\infty(0,T_f; S)$.
\end{proof}


\section{Fully-discrete approximations}
\label{sec:3}

In this section, we present the fully-discrete formulation for the coupled problem \eqref{fourfield}. Let $\mathcal{T}_h$ be a shape-regular, quasi-uniform triangulation of the domain $\Omega$ into disjoint elements $E$, where $h = \max_{E \in \mathcal{T}_h} \text{diam}(E)$ denotes the global mesh size parameter. 

\subsection{Spatial and temporal discretization}
\label{STd}

For the spatial discretization, we approximate the mechanics pair $(\bm u, \xi)$ using stable Taylor-Hood elements, consisting of continuous piecewise polynomials of degree $k\geq 2$ for displacement and degree $k-1$ for total pressure. For the transport variables $(\phi, \psi)$, we use continuous piecewise polynomials of degree $l \ge 1$. Accordingly, the finite element spaces are defined as:
\begin{align*}
    & \bm{V}_h := \{ \bm{v}_h \in [C^0(\overline{\Omega})]^d \cap \bm{V} : \bm{v}_h|_E \in [\mathbb{P}_k(E)]^d, \ \forall E \in \mathcal{T}_h \}, \\
    & W_h := \{ w_h \in C^0(\overline{\Omega}) \cap W : w_h|_E \in \mathbb{P}_{k-1}(E), \ \forall E \in \mathcal{T}_h \}, \\
    & Q_h = S_h := \{ q_h \in C^0(\overline{\Omega}) \cap Q : q_h|_E \in \mathbb{P}_l(E), \ \forall E \in \mathcal{T}_h \}. 
\end{align*}

For the temporal discretization, we partition the time interval $(0, T_f]$ into $N$ uniform sub-intervals of size $\Delta t = T_f/N$, defined by $0 = t_0 < t_1 < \dots < t_N = T_f$, where $t_n = n\Delta t$. 
For a generic discrete variable $\chi_h$, we denote its value at $t_n$ by $\chi_h^n$. For simplicity, we adopt the Backward Euler method, noting that the extension to higher-order schemes such as Crank-Nicolson or BDF2 is straightforward \cite{gu2024crank,gu2025pod}. Accordingly, we define the backward difference operator as:
\begin{align*}
    \bar{\partial}_t \chi_h^n := \frac{\chi_h^n - \chi_h^{n-1}}{\Delta t}.
\end{align*} 

To derive the a priori error estimates, we introduce the following assumption. We consider the spatial discretization for given values of $k$ and $l$. The following regularity assumption on the exact solution is imposed:
\begin{assumption}
\label{ass2}
    We assume that the exact solutions $(\bm{u},\xi,\phi,\psi)$ of System \eqref{weakform} satisfy
\begin{align*}
    & \bm{u} \in C (0,T_f; [H^1_{0,\Gamma_D} (\Omega) \cap W^{1,\infty}(\Omega)]^d ) \cap H^1(0,T_f;[H^{k+1}(\Omega)]^d) \cap H^2(0,T_f;[H^1(\Omega)]^d), 
    \\
    & \xi \in C(0,T_f; L^{2}(\Omega) \cap W^{1,\infty}(\Omega)) \cap H^1(0,T_f;H^k(\Omega)) \cap H^2(0,T_f;L^2(\Omega)), \\
    & \phi, \psi \in C(0,T_f; H^1_{0} (\Omega) \cap W^{1,\infty}(\Omega)) \cap H^1(0,T_f;H^{l+1}(\Omega)) \cap H^2(0,T_f;H^1(\Omega)).
\end{align*}
\end{assumption}

Note that the Taylor-Hood pair $(\bm{V}_h, W_h)$ satisfies the discrete inf-sup condition. That is, there exists a constant $\tilde{C}_{is} > 0$, independent of the mesh size $h$, such that
\begin{align} \label{disc:infsup}
    \tilde{C}_{is} \| w_h \| \leq \sup_{\bm 0 \not = \bm{v}_h \in \bm{V}_h} \frac{(\nabla \cdot \bm{v}_h, w_h)}{\| \bm{v}_h \|_{H^1}}  \quad \forall w_h \in W_h.
\end{align}
We define the Stokes-like projection $(\Pi_h^{\bm{u}} \bm{u}, \Pi_h^\xi \xi) \in \bm{V}_h \times W_h$ satisfying
\begin{subequations}
\label{eq:stokes_projection}
\begin{align}
(2\mu \varepsilon(\Pi_h^{\bm{u}} \bm{u}), \varepsilon(\bm{v}_h)) + (\Pi_h^\xi \xi, \nabla \cdot \bm{v}_h) &= (2\mu \varepsilon(\bm{u}), \varepsilon(\bm{v}_h)) + (\xi, \nabla \cdot \bm{v}_h), \\
(\nabla \cdot \Pi_h^{\bm{u}} \bm{u}, w_h) &= (\nabla \cdot \bm{u}, w_h),
\end{align}
\end{subequations}
for all $(\bm v_h, w_h) \in \bm{V}_h \times W_h$. The Stokes-like projection satisfies:
\begin{align}
\| \bm{u} - \Pi_h^{\bm{u}} \bm{u} \|_{H^1} + \| \xi - \Pi_h^\xi \xi \| \leq C h^k \left( \| \bm{u} \|_{H^{k+1}} + \| \xi \|_{H^k} \right). \label{interpolation11}
\end{align}
We also introduce the joint elliptic projection $(\Pi_h^\phi \phi, \Pi_h^\psi \psi) \in Q_h \times S_h$ satisfying
\begin{align}
\label{eq:elliptic_projection}
& K(\nabla \Pi_h^\phi \phi, \nabla q_h) + D(\nabla \Pi_h^\psi \psi, \nabla s_h) + \gamma(\Pi_h^\phi \phi - \Pi_h^\psi \psi, q_h - s_h) \nonumber \\
= & \ K(\nabla \phi, \nabla q_h) + D(\nabla \psi, \nabla s_h) + \gamma(\phi - \psi, q_h - s_h),
\end{align}
for all $(\bm v_h, w_h) \in Q_h \times S_h$. The elliptic projection satisfies:
\begin{align}
& \| \phi - \Pi_h^\phi \phi \|_{H^1} + \| \psi - \Pi_h^\psi \psi \|_{H^1} \leq C h^l \left( \| \phi \|_{H^{l+1}} + \| \psi \|_{H^{l+1}} \right), \label{interpolation22} \\
& \| \phi - \Pi_h^\phi \phi \| + \| \psi - \Pi_h^\psi \psi \| \leq C h^{l+1} \left( \| \phi \|_{H^{l+1}} + \| \psi \|_{H^{l+1}} \right).
\end{align}

\subsection{Fully-discrete monolithic algorithm}

Based on the spatial and temporal discretizations defined above, we formulate the fully-discrete problem. The solution process follows a time-stepping strategy initialized by the known conditions at $t=0$, denoted by $(\bm{u}^{0}_h, \xi^{0}_h, \phi^{0}_h, \psi^{0}_h)$. Given the solution $(\bm{u}^{n-1}_h, \xi^{n-1}_h, \phi^{n-1}_h, \psi^{n-1}_h)$ from the previous time step, we advance the system to the current time level $t_n$. For each time step $n=1, \dots, N$, the problem reads:

Find $(\bm{u}^n_h, \xi^n_h, \phi^n_h, \psi^n_h) \in \bm{V}_h \times W_h \times Q_h \times S_h$ such that:
\begin{subequations}
\label{discreteform}
\begin{align}
2\mu(\varepsilon(\bm u^n_h),\varepsilon(\bm v_h))
- (\xi^n_h,\nabla\cdot\bm v_h)
& =  (\bm f^n,\bm v_h) \ \ \forall \bm{v}_h \in \bm{V}_h, \label{dis:mon1}
\\
(\nabla\cdot\bm u^n_h, w_h)
+\tfrac{1}{\lambda} (\xi^n_h, w_h)
-\tfrac{\alpha}{\lambda} (\phi^n_h, w_h)
-\tfrac{\beta}{\lambda} (\psi^n_h, w_h)
& = 0 \qquad \quad \ \forall w_h \in W_h,  \label{dis:mon2}
\\
(c_1+\tfrac{\alpha^2}{\lambda})( \bar{\partial}_t \phi^n_h , q_h)
+(\tfrac{\alpha\beta}{\lambda}-b_0)( \bar{\partial}_t \psi^n_h , q_h)
-\tfrac{\alpha}{\lambda}( \bar{\partial}_t \xi^n_h , q_h & ) \nonumber \\ 
+ K(\nabla \phi^n_h, \nabla q_h) + \gamma (\phi^n_h-\psi^n_h, q_h)
+ \eta_{\phi} h^2 ( \nabla \bar{\partial}_t \phi^n_h, \nabla q_h) & = (g^n,q_h)  \quad \forall q_h \in Q_h,  \label{dis:mon3}
\\
(c_2+\tfrac{\beta^2}{\lambda})(\bar{\partial}_t \psi^n_h, s_h)
+(\tfrac{\alpha\beta}{\lambda}-b_0)( \bar{\partial}_t \phi^n_h , s_h)
-\tfrac{\beta}{\lambda}(\bar{\partial}_t \xi^n_h, s_h & ) 
\nonumber \\
+ D(\nabla \psi^n_h, \nabla s_h) + \gamma (\psi^n_h-\phi^n_h, s_h)
+ \eta_{\psi} h^2 ( \nabla \bar{\partial}_t \psi^n_h , \nabla s_h)  & = (h^n,s_h) \quad \forall s_h \in S_h,  \label{dis:mon4}
\end{align}
\end{subequations}
where the stabilization terms $\eta_\phi h^2(\nabla \bar{\partial}_t \phi_h^n, \nabla q_h)$ and $\eta_\psi h^2(\nabla \bar{\partial}_t \psi_h^n , \nabla s_h)$ are included to dampen potential non-physical oscillations in the generalized pressures \cite{rodrigo2016stability}. We choose parameters $\eta_\phi, \eta_\psi \geq 0$, which ensure consistency with the continuous problem.

We now proceed to the error analysis. We decompose the errors into approximation and discrete components as follows:
\[
\bm{e}_{u}^n = \bm{u}^n - \bm{u}_h^n = (\bm{u}^n - \Pi_h^{\bm{u}} \bm{u}^n) + (\Pi_h^{\bm{u}} \bm{u}^n - \bm{u}_h^n) =: \bm e_{u}^{I,n} + \bm e_{u}^{h,n}, 
\]
for the velocity vector $\bm u$, and similarly for the scalar variables $\chi \in \{\xi, \phi, \psi\}$,
\[
e_{\chi}^n = \chi^n - \chi_h^n = (\chi^n - \Pi_h^{\chi} \chi^n) + (\Pi_h^{\chi} \chi^n - \chi_h^n) =: e_{\chi}^{I,n} + e_{\chi}^{h,n}. 
\]
Using the interpolation properties, the approximation errors with superscript $I$ are bounded. Therefore, the following analysis focuses on estimating the discrete errors with superscript $h$. In the following estimates, we use the notation $X \lesssim Y$ to denote $X \leq C Y$, where $C > 0$ represents a generic constant independent of the mesh size $h$ and the time step $\Delta t$.

\begin{lemma}
\label{lem:truncation_error}
Let $(\bm u, \xi, \phi, \psi)$ and $(\bm u_h^n, \xi_h^n, \phi_h^n, \psi_h^n)$ be the solutions of system \eqref{weakform} and \eqref{discreteform}, respectively. Under Assumption \ref{ass2}, there holds:
\begin{align}
& \mu \| \bm e_u^{h,N} \|_{H^1}^2 + \Delta t \sum_{n=1}^N \mathcal{V}(e_\phi^{h,n}, e_\psi^{h,n}) 
\nonumber \\
& \lesssim (\Delta t)^2  + h^{2k} + h^{2l+2} + h^2 \big( \eta_\phi \max_{1\leq n \leq N} \| \partial_t\phi^n\|^2_{H^1}  + \eta_\psi \max_{1\leq n \leq N} \| \partial_t\psi^n\|^2_{H^1}  \big), 
\\
& \Delta t \sum_{n=1}^N \mathcal{E}(\bar{\partial}_t \bm e_u^{h,n}, \bar{\partial}_t e_\xi^{h,n}, \bar{\partial}_t e_\phi^{h,n}, \bar{\partial}_t e_\psi^{h,n}) + K\| e_\phi^{h,N} \|_{H^1}^2  +D \| e_\psi^{h,N} \|_{H^1}^2
\nonumber \\
& \lesssim (\Delta t)^2  + h^{2k} + h^{2l+2} + h^2 \big( \eta_\phi \max_{1\leq n \leq N} \| \partial_t\phi^n\|^2_{H^1}  + \eta_\psi \max_{1\leq n \leq N} \| \partial_t\psi^n\|^2_{H^1}  \big).
\end{align}
Furthermore, the error term $e_\xi^{h,N}$ satisfies:
\begin{align}
    \tilde{C}_{is} \| e_\xi^{h,N} \| & \lesssim  \mu \| \bm e_u^{h,N} \|_{H^1}.    \label{lem:te3}
\end{align}
\end{lemma}

The proof is deferred to Appendix \ref{appendix}. With the discrete errors bounded, we combine them with the standard interpolation estimates to obtain the main convergence result.

\begin{theorem}
\label{thm:main_convergence}
Let $(\bm u, \xi, \phi, \psi)$ and $(\bm u_h^n, \xi_h^n, \phi_h^n, \psi_h^n)$ be the solutions of system \eqref{weakform} and \eqref{discreteform}, respectively. Under Assumption \ref{ass2}, there holds:

\begin{align}
& \mu \| \bm e_{u}^{h,N} \|_{H^1}^2 + \| e_\xi^{h,N} \|^2 
 + K \| e_\phi^{h,N} \|_{H^1}^2 + D \| e_\psi^{h,N} \|_{H^1}^2
\nonumber \\
& \lesssim (\Delta t)^2  + h^{2k} + h^{2l} + h^2 \big( \eta_\phi \max_{1\leq n \leq N} \| \partial_t\phi^n\|^2_{H^1}  + \eta_\psi \max_{1\leq n \leq N} \| \partial_t\psi^n\|^2_{H^1}  \big).
\end{align}
\end{theorem}

\begin{proof}
The proof follows directly from the triangle inequality. 
The approximation errors are bounded by the interpolation properties \eqref{interpolation11} and \eqref{interpolation22} (bounded by $O(h^{2k})$ and $O(h^{2l})$). The discrete errors are bounded by Lemma \ref{lem:truncation_error}. Summing these contributions yields the result.
\end{proof}

\subsection{Global-in-time iterative decoupled algorithm}

Although direct solution of the monolithic system \eqref{discreteform} is possible, it can be computationally expensive. To address this, iterative schemes are often applied, such as the time-stepping methods found in \cite{cai2025efficient, zhao2025optimally}. In this section, however, we focus on a ``global-in-time" approach based on the continuous splitting defined in \eqref{iter:step1}--\eqref{iter:step2}. Specifically, given the solution history from the previous iteration, we compute the new approximation by solving two decoupled sub-problems sequentially over the full time trajectory.

Let $i$ denote the iteration index. For a generic discrete variable $\chi_h$, we denote its value at time $t_n$ and iteration $i$ by $\chi_h^{n,i}$. For the initial time step $n=0$, we fix the iterative solution to match the initial conditions: $(\bm{u}^{0,i}_h, \xi^{0,i}_h, \phi^{0,i}_h, \psi^{0,i}_h) = (\bm{u}^{0}_h, \xi^{0}_h, \phi^{0}_h, \psi^{0}_h)$ for all iterations $i \ge 0$.

\paragraph{Step 1: Discrete transport sub-problem}
Given the total pressure from the previous iteration $\{ \xi^{n,i-1}_h\}_{n=1}^N$, we solve for the sequence of variables $\{(\phi_h^{n,i}, \psi_h^{n,i})\}_{n=1}^N \subset Q_h \times S_h$ such that:
\begin{subequations}
\label{iter:dis1}
\begin{align}
(c_1+\tfrac{\alpha^2}{\lambda})( \bar{\partial}_t \phi^{n,i}_h , q_h) + (\tfrac{\alpha\beta}{\lambda}-b_0)( \bar{\partial}_t \psi^{n,i}_h, q_h)  + K(\nabla\phi^{n,i}_h,\nabla q_h) + \gamma(\phi^{n,i}_h- & \psi^{n,i}_h, q_h) 
\nonumber \\
+ \eta_\phi h^2 ( \nabla \bar{\partial}_t \phi^{n,i}_h ,\nabla q_h) 
 = (g^n,q_h) + \tfrac{\alpha}{\lambda}(\bar{\partial}_t \xi^{n,i-1}_h , q_h) \quad \forall q_h \in Q_h, & \label{iter:disc_phi} \\
(c_2+\tfrac{\beta^2}{\lambda})( \bar{\partial}_t \psi^{n,i}_h , s_h) + (\tfrac{\alpha\beta}{\lambda}-b_0) ( \bar{\partial}_t \phi^{n,i}_h , s_h)  + D(\nabla\psi^{n,i}_h,\nabla s_h) + \gamma(\psi^{n,i}_h- &\phi^{n,i}_h, s_h) 
\nonumber \\  
+ \eta_\psi h^2 ( \nabla \bar{\partial}_t \psi^{n,i}_h ,\nabla s_h) 
 = (h^n,s_h) + \tfrac{\beta}{\lambda}( \bar{\partial}_t \xi^{n,i-1}_h, s_h) \quad \forall s_h \in S_h. & 
 \label{iter:disc_psi}
\end{align}
\end{subequations}
Note that while the variables are coupled at each specific time step $t_n$, the time-stepping nature of \eqref{iter:dis1} means the system is solved sequentially from $n=1$ to $N$.

\paragraph{Step 2: Discrete mechanics sub-problem}
Using the updated transport variables $\{(\phi_h^{n,i}, \psi_h^{n,i})\}_{n=1}^N$ obtained in \textit{Step 1}, we solve for the sequence of mechanics variables $\{(\bm{u}_h^{n,i}, \xi_h^{n,i})\}_{n=1}^N \subset \bm{V}_h \times W_h$ satisfying:
\begin{subequations}
\label{iter:dis2}
\begin{align}
& 2\mu(\varepsilon(\bm u_h^{n,i}),\varepsilon(\bm v_h)) - (\xi_h^{n,i},\nabla\cdot\bm v_h) = (\bm f^n, \bm v_h) & \forall \bm{v}_h \in \bm{V}_h,  \label{iter:disc_u} 
\\
& (\nabla\cdot\bm u_h^{n,i}, w_h) +\tfrac{1}{\lambda} (\xi_h^{n,i}, w_h) = \tfrac{\alpha}{\lambda} (\phi_h^{n,i}, w_h) + \tfrac{\beta}{\lambda} (\psi_h^{n,i}, w_h) & \forall w_h \in W_h. \label{iter:disc_xi}
\end{align}
\end{subequations}
Since the mechanics sub-problem \eqref{iter:dis2} is quasi-static (no time derivatives), the steps are decoupled and can be solved in parallel.

\vspace{0.5cm}

This iterative procedure repeats until the relative error between consecutive iterations satisfies a prescribed tolerance or a maximum number of iterations is reached.
To analyze the convergence, we define the discrete errors between the sequence $\{(\bm u_h^{n, i}, \xi_h^{n, i}, \phi_h^{n, i}, \psi_h^{n, i})\}$ obtained from \eqref{iter:dis1}--\eqref{iter:dis2} and the monolithic discrete solution $(\bm u_h^n, \xi_h^n, \phi_h^n, \psi_h^n)$ of \eqref{discreteform} as:
\begin{align}
\label{errordef}
\bm e_{u}^{n,i} = \bm{u}_h^{n,i} - \bm{u}_h^n, \quad e_{\xi}^{n,i} = \xi_h^{n,i} - \xi_h^n, \quad
e_{\phi}^{n,i} = \phi_h^{n,i} - \phi_h^n, \quad e_{\psi}^{n,i} = \psi_h^{n,i} - \psi_h^n.
\end{align}
Analogous to the continuous case, the following theorem establishes the convergence of the fully discrete scheme, adapting the analysis by replacing continuous time derivatives with backward difference quotients and integrals with finite summations.

\begin{theorem}
The fully discrete iterative scheme \eqref{iter:dis1}--\eqref{iter:dis2} is strictly contractive. Specifically, the following discrete error estimate holds:
\begin{align}
\bigg( \Delta t \sum_{n=1}^N \| \bar{\partial}_t e_\xi^{n,i} \|^2 \bigg)^{1/2} \le L_{\mathrm{dis}} \bigg( \Delta t \sum_{n=1}^N \| \bar{\partial}_t e_\xi^{n,i-1} \|^2 \bigg)^{1/2}, \label{dis:main1}
\end{align}
where the discrete contraction factor is given by
\begin{align}
    L_{\mathrm{dis}} := \frac{1}{\sqrt{(1+2\lambda C_{\tilde{\mu}})(1 + 2\lambda h^2 C_{\mathrm{stab}})}} < 1, \label{def:Ldis}
\end{align}
with $C_{\tilde{\mu}},  C_{\mathrm{stab}}$ being positive constants independent of mesh size $h$ and time step size $\Delta t$.
As $i \to \infty$, the sequence of discrete iterates converges to the monolithic discrete solution as:
\begin{align*}
    \max_{1 \le n \le N} \|\bm{e}_{u}^{n,i}\|_{H^1} &\to 0, \quad
    \max_{1 \le n \le N} \|e_{\xi}^{n,i}\| \to 0, 
    \\
    \max_{1 \le n \le N} \|e_{\phi}^{n,i}\|_{H^1} &\to 0, \quad
    \max_{1 \le n \le N} \|e_{\psi}^{n,i}\|_{H^1} \to 0.
\end{align*}
Consequently, the discrete iterative scheme converges globally.
\end{theorem}

\begin{proof}
We follow the structure of the proof of the continuous case.

\noindent
\textbf{1. Discrete transport error equations.}
Subtracting the monolithic equations \eqref{dis:mon1}--\eqref{dis:mon2} from the discrete transport sub-problem \eqref{iter:dis1} yields the error equations:
\begin{align}
& (c_1+\tfrac{\alpha^2}{\lambda})( \bar{\partial}_t e_\phi^{n,i} , q_h) + (\tfrac{\alpha\beta}{\lambda}-b_0)( \bar{\partial}_t e_\psi^{n,i}, q_h)  + K(\nabla e_\phi^{n,i},\nabla q_h) \nonumber \\
& + \gamma(e_\phi^{n,i}-e_\psi^{n,i}, q_h) 
+ \eta_\phi h^2 ( \nabla \bar{\partial}_t e_\phi^{n,i} ,\nabla q_h) 
 = \tfrac{\alpha}{\lambda}(\bar{\partial}_t e_\xi^{n,i-1}, q_h), \label{err:disc_phi} \\
& (c_2+\tfrac{\beta^2}{\lambda})( \bar{\partial}_t e_\psi^{n,i} , s_h) + (\tfrac{\alpha\beta}{\lambda}-b_0) ( \bar{\partial}_t e_\phi^{n,i} , s_h)  + D(\nabla e_\psi^{n,i},\nabla s_h) \nonumber \\ 
& + \gamma(e_\psi^{n,i}-e_\phi^{n,i}, s_h) 
+ \eta_\psi h^2 ( \nabla \bar{\partial}_t e_\psi^{n,i} ,\nabla s_h) 
 = \tfrac{\beta}{\lambda}( \bar{\partial}_t e_\xi^{n,i-1}, s_h). \label{err:disc_psi}
\end{align}
We choose the test functions $q_h=\bar{\partial}_t e_\phi^{n,i}$ in \eqref{err:disc_phi} and $s_h=\bar{\partial}_t e_\psi^{n,i}$ in \eqref{err:disc_psi}. Summing the two equations yields:
\begin{align}
& (c_1-b_0)  \| \bar{\partial}_t e_\phi^{n,i} \|^2 + (c_2-b_0) \|  \bar{\partial}_t e_\psi^{n,i} \|^2
+ b_0 \| \bar{\partial}_t e_\phi^{\,n,i} - \bar{\partial}_t e_\psi^{\,n,i}\|^2 + \tfrac{1}{\lambda} \| \alpha \bar{\partial}_t  e_{\phi}^i +  \beta \bar{\partial}_t {e}_\psi^i \|^2 \nonumber\\
& + K 
(\nabla e_\phi^{n,i},\nabla \bar{\partial}_t e_\phi^{n,i} )
 + D ( \nabla e_\psi^{n,i} , \nabla \bar{\partial}_t e_\psi^{n,i} ) + \gamma  (e_\phi^{n,i}-e_\psi^{n,i}, \bar{\partial}_t e_\phi^{n,i} - \bar{\partial}_t e_\psi^{n,i})
\nonumber\\
& + \eta_\phi h^2 \| \nabla \bar{\partial}_t e_\phi^{n,i} \|^2 + \eta_\psi h^2 \| \nabla \bar{\partial}_t e_\psi^{n,i} \|^2  = \tfrac{1}{\lambda} (  \bar{\partial}_t e_\xi^{n,i-1}, \alpha  \bar{\partial}_t e_\phi^{n,i}+\beta  \bar{\partial}_t e_\psi^{n,i} ).
\label{disc:transport_step_energy}
\end{align}
Using Cauchy--Schwarz and Young inequalities on the right-hand side of \eqref{disc:transport_step_energy}, we have:
\begin{align}
    \tfrac{1}{\lambda} ( \bar{\partial}_t e_\xi^{n,i-1}, \alpha  \bar{\partial}_t e_\phi^{n,i}+\beta  \bar{\partial}_t e_\psi^{n,i} ) \leq \tfrac{1}{2\lambda} \|\bar{\partial}_t e_\xi^{n,i-1}\|^2 + \tfrac{1}{2\lambda} \| \alpha  \bar{\partial}_t e_\phi^{n,i} +\beta  \bar{\partial}_t e_\psi^{n,i} \|^2. \label{disc:transport_step_energy_ieq}
\end{align}
Under Assumption \ref{assumption1}, the first three terms of \eqref{disc:transport_step_energy} are non-negative and can therefore be dropped.
For the diffusion and coupling terms, we apply the algebraic identity $( x, x-y) = \tfrac12 (\|x\|^2-\|y\|^2 + \|x-y\|^2 )$ and drop the non-negative dissipation terms $\|x-y\|^2$ to derive 
\begin{align}
& \tfrac{1}{2\lambda} \| \alpha \bar{\partial}_t  e_{\phi}^i +  \beta \bar{\partial}_t {e}_\psi^i \|^2 + \eta_\phi h^2 \| \nabla \bar{\partial}_t e_\phi^{n,i} \|^2 + \eta_\psi h^2 \| \nabla \bar{\partial}_t e_\psi^{n,i} \|^2 
\nonumber \\ 
& \quad + \tfrac{K}{2\Delta t} \| \nabla e_\phi^{n,i} \|^2 + \tfrac{D}{2 \Delta t} \| \nabla e_\psi^{n,i} \|^2 + \tfrac{\gamma}{2 \Delta t} \|e_\phi^{n,i}-e_\psi^{n,i}\|^2
\nonumber \\ 
& \leq \tfrac{1}{2\lambda} \|\bar{\partial}_t e_\xi^{n,i-1}\|^2 + \tfrac{K}{2 \Delta t} \| \nabla e_\phi^{n-1,i} \|^2 + \tfrac{D}{2 \Delta t} \| \nabla e_\psi^{n-1,i} \|^2 + \tfrac{\gamma}{2 \Delta t} \|e_\phi^{n-1,i}-e_\psi^{n-1,i}\|^2. \label{dis:ineq1}
\end{align}
We note that the stabilization terms provide additional control. Specifically, there exists a constant $C_{\text{stab}} = C_{\text{stab}}(\alpha,\beta,\eta_\phi,\eta_\psi) \ge 0$ such that 
\begin{align}
C_{\text{stab}} \| \alpha \bar{\partial}_t  e_{\phi}^i +  \beta \bar{\partial}_t {e}_\psi^i \|^2 \leq \eta_\phi \| \nabla \bar{\partial}_t e_\phi^{n,i} \|^2 + \eta_\psi \| \nabla \bar{\partial}_t e_\psi^{n,i} \|^2.  \label{dis:ineq2}
\end{align}
Combing the inequalities \eqref{dis:ineq1} and \eqref{dis:ineq2}, we then apply the summation operator $2 \Delta t \sum_{n=1}^N$. Note that the diffusion and coupling terms form telescoping sums; since the initial errors are zero, these sums collapse, leaving only the contributions at the final time step $N$. This yields the global bound:
\begin{align}
    & K \| \nabla e_\phi^{N,i} \|^2 + D \| \nabla e_\psi^{N,i} \|^2 + \gamma \|e_\phi^{N,i}-e_\psi^{N,i}\|^2   \nonumber \\
    & + ( 2 h^2 C_{\text{stab}} + \tfrac{1}{\lambda} ) \Delta t \sum_{n=1}^N \| \alpha \bar{\partial}_t  e_{\phi}^i + \beta \bar{\partial}_t {e}_\psi^i \|^2 \leq \tfrac{1}{\lambda} \Delta t \sum_{n=1}^N \|\bar{\partial}_t e_\xi^{n,i-1}\|^2.
    \label{disc:transport_sum}
\end{align}

\noindent
\textbf{2. Discrete mechanics error equations.} Subtracting the monolithic equations \eqref{dis:mon3}--\eqref{dis:mon4} from the discrete mechanics sub-problem \eqref{iter:dis2} and applying the backward difference operator $\bar{\partial}_t(\cdot)$ to the resulting error equations yields:
\begin{align}
& 2\mu(\varepsilon(\bar{\partial}_t \bm e_u^{n,i}),
\varepsilon(\bm v_h))
- (\bar{\partial}_t e_\xi^{n,i},
\nabla\cdot \bm v_h)=0, \label{disc:mech1}
\\
& (\nabla\cdot \bar{\partial}_t \bm e_u^{n,i},w_h)
+ \tfrac{1}{\lambda}(\bar{\partial}_t e_\xi^{n,i},w_h)
= \tfrac{1}{\lambda}(\alpha\bar{\partial}_t e_\phi^{n,i}
+ \beta\bar{\partial}_t e_\psi^{n,i},w_h). \label{disc:mech2}
\end{align}
We choose $\bm v_h=\bar{\partial}_t \bm e_u^{n,i}$ in \eqref{disc:mech1} and $w_h=\bar{\partial}_t e_\xi^{n,i}$ in  \eqref{disc:mech2}. Adding the two equations and applying the Cauchy--Schwarz and Young inequalities, we get:
\begin{align} 
2\mu\|\varepsilon(\bar{\partial}_t \bm e_u^{n,i})\|^2
+ \tfrac{1}{\lambda}\|\bar{\partial}_t e_\xi^{n,i}\|^2 & = \tfrac{1}{\lambda}(\alpha\bar{\partial}_t e_\phi^{n,i}
+ \beta\bar{\partial}_t e_\psi^{n,i},w_h) \nonumber \\
& \le \tfrac{1}{2\lambda}  \|\bar{\partial}_t {e}_{\xi}^{i}\|^2 + \tfrac{1}{2\lambda} \|\alpha\bar{\partial}_t {e}_\phi^i + \beta\bar{\partial}_t {e}_\psi^i\|^2.
\label{dis:eq6}
\end{align}
Using the discrete inf-sup condition \eqref{disc:infsup} to \eqref{disc:mech1} yields
\begin{align*}
\tilde{C}_{is} \| \bar{\partial}_t  e_\xi^{n,i} \| &\leq \sup_{\bm 0 \not = \bm{v}_h\in \bm{V}_h} \frac{(\bar{\partial}_t  e_\xi^{n,i}, \nabla \cdot \bm{v}_h)}{\| \bm{v}_h \|_{H^1}} 
= \sup_{\bm 0 \not = \bm{v}_h\in \bm{V}_h} \frac{2\mu (\varepsilon(\bar{\partial}_t {\bm e}_u^{n,i}), \varepsilon(\bm{v}_h))}{\| \bm{v}_h \|_{H^1}} 
\leq 2\mu \| \varepsilon(\bar{\partial}_t {\bm e}_u^{n,i}) \|,
\end{align*}
implying that there exists \( C_{\tilde{\mu}} >0\) (independent of \(h\) and \(\Delta t\)) such that
\begin{align}
 C_{\tilde{\mu}} \|\bar{\partial}_t  e_\xi^{n,i}\|^2 \le 2\mu\|\varepsilon(\bar{\partial}_t {\bm e}_u^{n,i})\|^2. \label{dis:xibound}
\end{align}
Substituting \eqref{dis:xibound} into \eqref{dis:eq6} and applying the summation operator $\Delta t \sum_{n=1}^N$ yields:
\begin{equation}
(2\lambda C_{\tilde{\mu}}+1)
\Delta t \sum_{n=1}^N
\|\bar{\partial}_t e_\xi^{n,i}\|^2
\le
\Delta t \sum_{n=1}^N
\|\alpha\bar{\partial}_t e_\phi^{n,i}
+ \beta\bar{\partial}_t e_\psi^{n,i}\|^2. \label{disc:mech_sum}
\end{equation}

\medskip
\noindent\textbf{3. Contraction and convergence.} Since $C_{\tilde{\mu}}, \lambda > 0$ and $C_{\mathrm{stab}} \ge 0$, we have $L_{\mathrm{dis}}$ defined in \eqref{def:Ldis} strictly smaller than $1$.
Combining \eqref{disc:transport_sum} and \eqref{disc:mech_sum}, we obtain the global inequality:
\begin{align}
\Delta t \sum_{n=1}^N
\|\bar{\partial}_t e_\xi^{n,i}\|^2
\le
L_{\mathrm{dis}}^2 \Delta t \sum_{n=1}^N
\|\bar{\partial}_t e_\xi^{n,i-1}\|^2, \label{disc:contraction}
\end{align}
establishing the contraction property \eqref{dis:main1}. Finally, the convergence of the error components $\bm e_u^{n,i}$, $e_\xi^{n,i}$, $e_\phi^{n,i}$, and $e_\psi^{n,i}$ follows directly from the stability bounds derived above using discrete summation arguments analogous to the continuous case.
\end{proof}

\section{Numerical experiments}
\label{sec:4}

In this section, we present numerical experiments to validate the theoretical analysis and performance of the proposed finite element schemes. We assess the convergence rates of the errors in the respective norms and verify the stability of the method. All computations are performed using FEniCSx with multiphenicsx \cite{baratta2023dolfinx, multiphenicsx}.

\subsection{Convergence verification}
\label{subsec:convergence}

To evaluate the accuracy of the proposed algorithms, we employ the method of manufactured solutions. We consider the unit square domain $\Omega = (0,1) \times (0,1)$ over the time interval $(0, T_f]$. Consistent with the boundary partition definitions in \eqref{boundaries}, we define the Neumann boundary $\Gamma_N:=\{(1,y); 0\leq y\leq1\}$ as the right edge of the domain and the Dirichlet boundary $\Gamma_D:=\Gamma \setminus \Gamma_N$ as the remaining part.
We construct the analytical solutions as follows:
\begin{subequations}
\label{eq:exact_solutions}
\begin{align}
    \bm{u}(x,y,t) &= 
    \begin{bmatrix}
    \sin{(\pi x t)} \cos{(\pi y t)} \\
    \cos{(\pi x t)} \sin{(\pi y t)}
    \end{bmatrix} x y (1-x)^2(1-y), 
    \\
    \phi(x,y,t) &= \cos{(t + x - y)} x y (1-x)^2(1-y),
    \\
    \psi(x,y,t) &= \sin{(t + x - y)} x y (1-x)^2(1-y).
\end{align}
\end{subequations}
The exact solution for the total pressure variable is derived consistently from the reformulation $\xi = -\lambda \nabla \cdot \bm{u} + \alpha \phi + \beta \psi$. 
The source terms $\bm{f}$, $g$, and $h$, as well as the initial and boundary conditions, are determined by substituting these expressions into the governing equations. We set $\eta_\phi = \eta_\psi = 0$. The physical parameters are set to:
\[
\mu = \lambda = K = D = \alpha = \beta = c_1 = c_2 = 1, \ \text{and} \quad b_0 = \gamma = 0.1.
\]
\paragraph{Verification of temporal convergence}
We first verify the temporal convergence rate of the monolithic scheme provided in \eqref{discreteform} by setting a fixed final time 
$T_f = 1.0$ and employing a sufficiently fine spatial mesh with $h=1/64$ and polynomial orders $k=l=3$. The simulations are performed using a sequence of time steps $\Delta t \in \{1/4, 1/8, 1/16, 1/32\}$. The results, summarized in Table \ref{tab:time_convergence}, confirm the expected first-order convergence $O(\Delta t)$ for all field variables, consistent with the backward Euler time discretization.

\vspace{-0.5cm}

\begin{table}[h]
\centering
\caption{Temporal convergence results at $t_N$ using a fixed spatial mesh.}
\label{tab:time_convergence}
\begin{tabular}{ccccccccc}
\toprule
$\Delta t$ & $\|\bm{u}^N - \bm{u}_h^N \|_{H^1}$ & Rate & $\|\xi^N  - \xi_h^N \|$ & Rate & $\|\phi^N - \phi_h^N\|_{H^1}$ & Rate & $\|\psi^N - \psi_h^N\|_{H^1}$ & Rate \\
\midrule
$1/4$  & $4.459 \times 10^{-4}$ & --    & $9.919 \times 10^{-4}$ & --    & $5.239 \times 10^{-3}$ & --    & $5.256 \times 10^{-3}$ & --      \\
$1/8$  & $2.389 \times 10^{-4}$ & 0.90  & $5.447 \times 10^{-4}$ & 0.87  & $2.777 \times 10^{-3}$ & 0.92  & $2.780 \times 10^{-3}$ & 0.92    \\
$1/16$ & $1.245 \times 10^{-4}$ & 0.94  & $2.878 \times 10^{-4}$ & 0.92  & $1.436 \times 10^{-3}$ & 0.95  & $1.436 \times 10^{-3}$ & 0.95    \\
$1/32$ & $6.365 \times 10^{-5}$ & 0.97  & $1.482 \times 10^{-4}$ & 0.96  & $7.312 \times 10^{-4}$ & 0.97  & $7.306 \times 10^{-4}$ & 0.98    \\
\bottomrule
\end{tabular}
\end{table}

\vspace{-1.0cm}

\paragraph{Verification of spatial convergence}
Next, we assess the spatial convergence properties of the monolithic scheme provided in \eqref{discreteform}. We fix a small final time $T_f = 0.01$ and use a sufficiently small time step $\Delta t = T_f/64$ to ensure that temporal errors are negligible. The polynomial orders are set to $k=l=2$. We compute the errors on a sequence of refined meshes with $h \in \{1/4, 1/8, 1/16, 1/32\}$. As shown in Table \ref{tab:spatial_convergence}, we observe optimal second-order convergence rates, $O(h^2)$, for all variables in their respective norms.

\begin{table}[h]
\centering
\caption{Spatial convergence results at $t_N$ using a fixed time step size.}
\label{tab:spatial_convergence}
\begin{tabular}{ccccccccc}
\toprule
$h$ & $\|\bm{u}^N - \bm{u}_h^N \|_{H^1}$ & Rate & $\|\xi^N - \xi_h^N \|$ & Rate & $\|\phi^N - \phi_h^N\|_{H^1}$ & Rate & $\|\psi^N - \psi_h^N\|_{H^1}$ & Rate \\
\midrule
$1/4$ & $5.610\times 10^{-4}$ & -- & $3.332\times 10^{-3}$ & -- & $5.914\times 10^{-3}$ & -- & $5.983\times 10^{-3}$ & -- \\
$1/8$ & $1.495\times 10^{-4}$ & 1.91 & $9.170\times 10^{-4}$ & 1.86 & $1.644\times 10^{-3}$ & 1.85 & $1.646\times 10^{-3}$ & 1.86 \\
$1/16$ & $3.757\times 10^{-5}$ & 1.99 & $2.341\times 10^{-4}$ & 1.97 & $4.189\times 10^{-4}$ & 1.97 & $4.190\times 10^{-4}$ & 1.97 \\
$1/32$ & $9.381\times 10^{-6}$ & 2.00 & $5.883\times 10^{-5}$ & 1.99 & $1.051\times 10^{-4}$ & 1.99 & $1.052\times 10^{-4}$ & 1.99 \\
\bottomrule
\end{tabular}
\end{table}

\paragraph{Convergence of the global-in-time iterative decoupled algorithm}

We now examine the convergence properties of the global-in-time iterative decoupled algorithm presented in \eqref{iter:dis1}--\eqref{iter:dis2} by setting the final time $T_f = 1.0$ with $\Delta t = T_f/32$ and employing the mesh size $h = 1/16$ and polynomial orders $k = l = 2$. To evaluate the convergence behavior, we compute the relative iterative errors, as defined in \eqref{errordef}, at the final time step $N$. Figure \ref{fig:iterative_convergence} illustrates the reduction of the iterative error with respect to the iteration count. We observe that the error decays monotonically toward zero as the iteration count increases, exhibiting a linear rate of convergence.

\begin{figure}[htbp]
    \centering
    \includegraphics[width=0.68\linewidth]{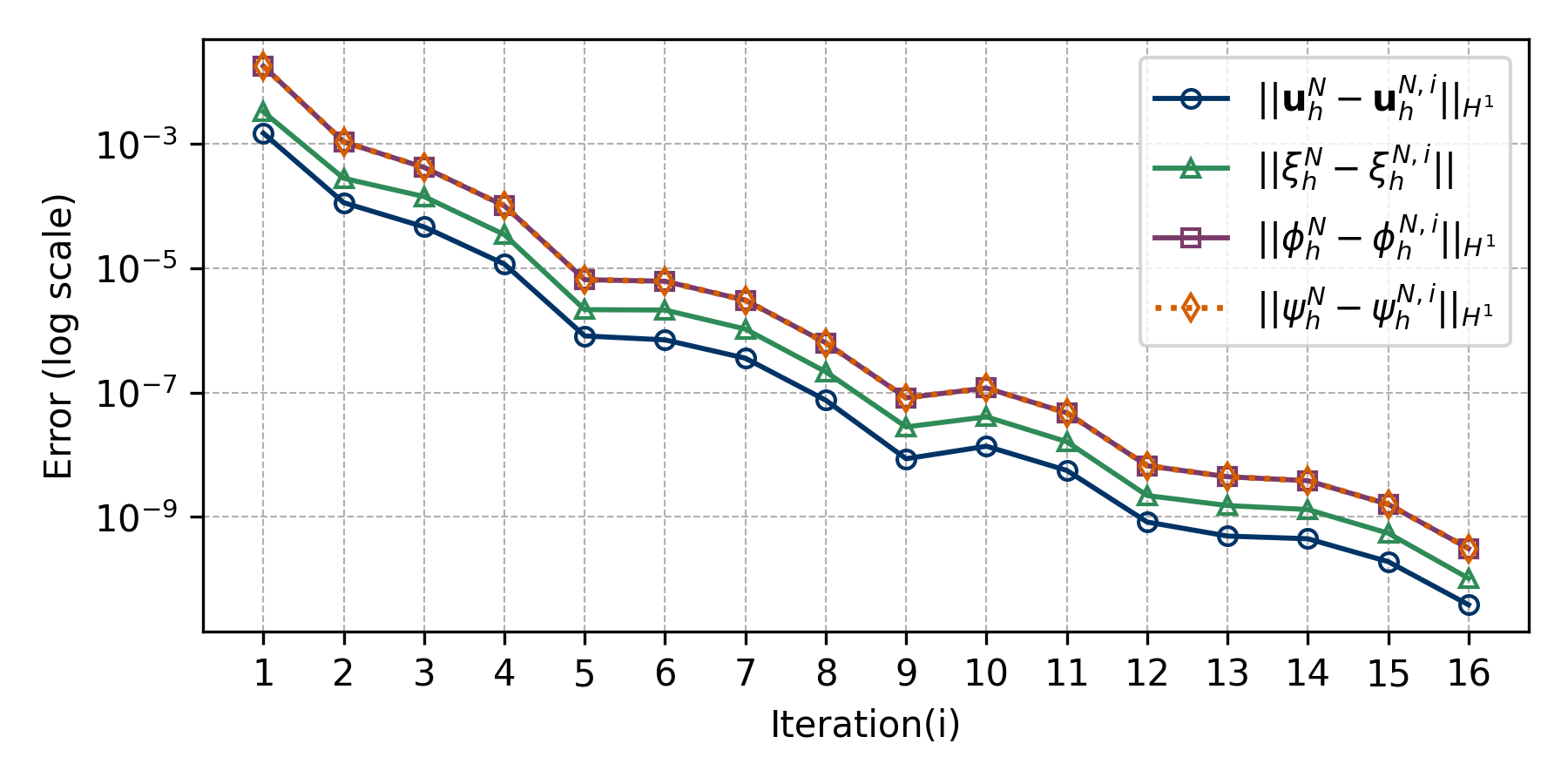}
    \caption{Convergence behavior of the relative iterative errors}
    \label{fig:iterative_convergence}
\end{figure}

\subsection{Simulation of the generalized Barry-Mercer problem}
\label{subsec:barry_mercer}

In this section, we consider a generalized version of the Barry-Mercer problem \cite{barry1999exact,cai2023some} to evaluate the robustness of the proposed schemes. The problem involves a  point source located at $(x_0,y_0) = (0.25, 0.25)$ in the unit square domain $\Omega = (0,1) \times (0,1)$. 
The boundary segments $\partial \Omega$ are: 
$\Gamma_1=\{(1,y); 0\leq y\leq1\}$, $\Gamma_2=\{(x,0);0\leq x\leq1\}$, $\Gamma_3=\{(0,y);0\leq y\leq1\}$, $\Gamma_4=\{(x,1);0\leq x\leq1\}$. Let $\bm{u}=[u_1,u_2]^T$. The boundary conditions on $\partial \Omega \times (0,T]$ are specified as follows:
\begin{align*}
    \frac{\partial u_1}{\partial x} = 0, \ u_2 = 0, \ \phi = 0, \ \psi = 0 \quad & \mbox{on} \ \Gamma_j \times (0,T], \ j=1,3,  \\
	\frac{\partial u_2}{\partial y} = 0, \ u_1 = 0, \ \phi = 0, \ \psi = 0  \quad & \mbox{on} \ \Gamma_j \times (0,T], \ j=2,4.
\end{align*}
The initial conditions are set to zero for all variables. 
The body force is set to $\bm{f} = \bm{0}$, while the source terms are defined as:
\begin{equation*}
    g(x,y, t) = h(x,y, t) = 2 \sin(\omega t) \delta(x - x_0) \delta(y - y_0).
\end{equation*}
where $\delta (\cdot)$ represents the Dirac function. We employ a uniform triangular mesh with $h = 1/64$ and polynomial orders $k=2, l=1$.

In the first test, we evaluate the performance of the schemes under the influence of the stabilization parameters. The material properties are selected as:
\[
E = 10^5, \quad \nu = 0.1, \quad \alpha_1 = \alpha_2 = 0.5, \quad K = D = 10^{-6}, \quad c_1 = c_2 = b_0 = \gamma = 0.0.
\] 
The simulation is performed for a single time step with $\Delta t = T_f = \frac{\pi}{2} \times 10^{-9}$. The global-in-time algorithm is configured with a fixed count of 30 iterations. We investigate the numerical results using the proposed stabilization scaling $\eta_\phi = \eta_\psi = [32(\mu+2\lambda)h^2]^{-1}$. The potentials $\phi$ and $\psi$ are identical under this configuration, as illustrated in Fig. \ref{fig:barry_mercer_results2}, no non-physical oscillations are observed for either the monolithic or iterative algorithms.

\begin{figure}[htbp]
    \centering
    \includegraphics[width=1\linewidth]{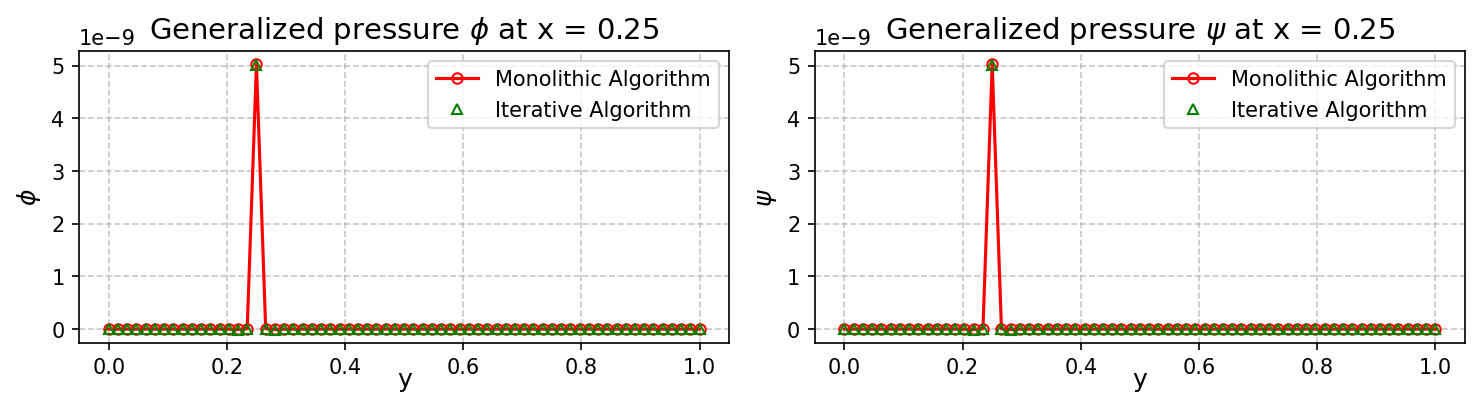}
    \caption{Cross-section of the generalized pressures $\phi$ and $\psi$ in the generalized Barry-Mercer problem.}
    \label{fig:barry_mercer_results2}
\end{figure}

In the second test, the physical parameters are chosen as follows:
$$\mu = 0.4, \quad \lambda = 0.2, \quad \alpha_1 = \alpha_2 = 0.5, \quad K = D = 1.0, \quad c_1 = c_2 = b_0 = \gamma = 0.0.$$
We employ a uniform triangular mesh with size $h = 1/64$ and polynomial orders $k = 2$, $l = 1$. The simulation runs to a final time $T_f = \pi/2$ with a time step size $\Delta t = T_f/20$. Following \cite{rodrigo2016stability}, we set the stabilization parameters as $\eta_\phi = \eta_\psi = \frac{1}{32(\mu+2\lambda)h^2}$.
Figure \ref{fig:barry_mercer_results} presents the numerical solutions along the vertical cross-section $x=0.25$ at the final time $t=\pi/2$. Here, the global-in-time algorithm uses a fixed number of 5 iterations. The numerical solution remains stable and convergent, with both algorithms showing excellent agreement with the exact solution.
\begin{figure}[htbp]
    \centering
    \includegraphics[width=1\linewidth]{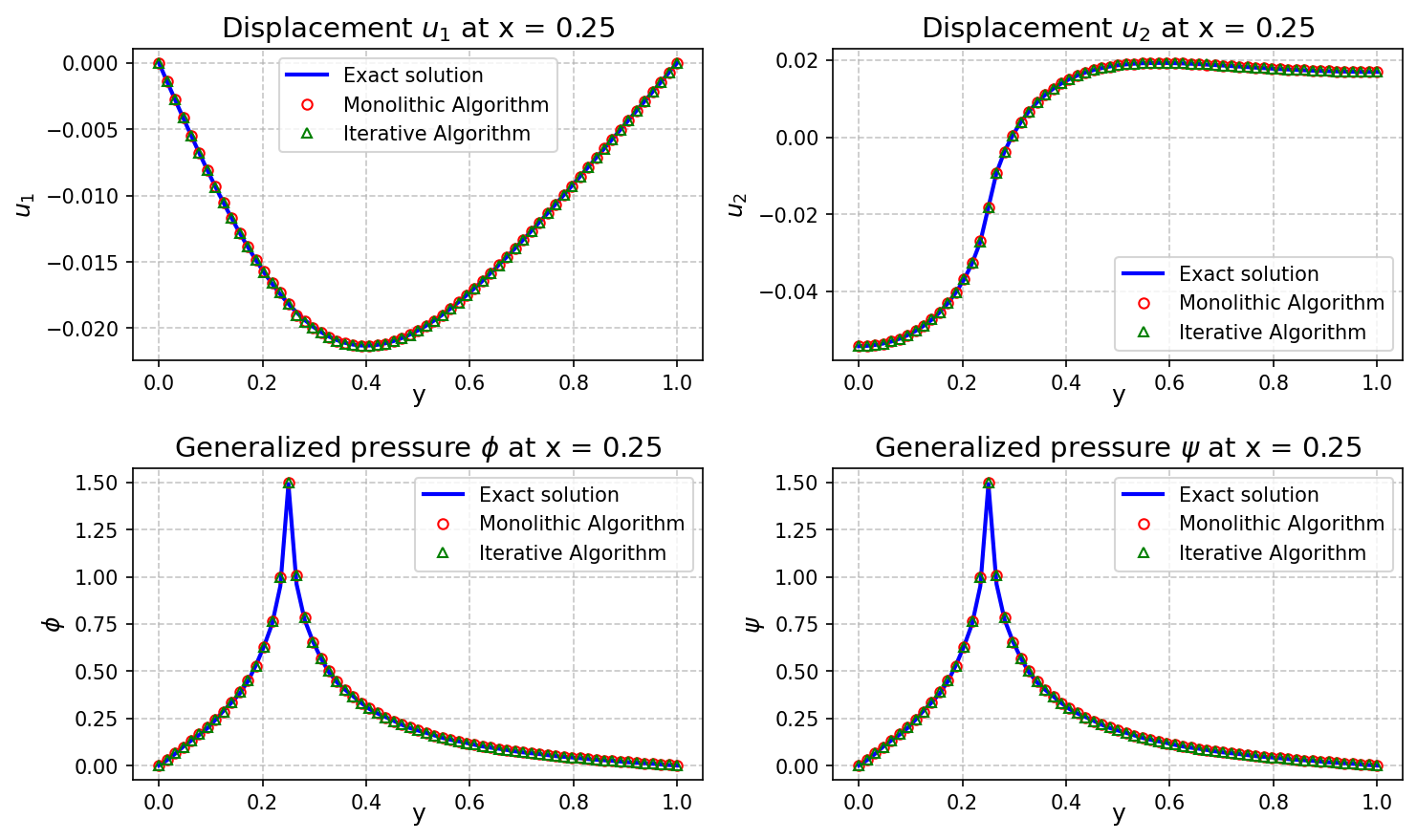}
    \caption{Numerical solution of the generalized Barry-Mercer problem at $t=\pi/2$.}
    \label{fig:barry_mercer_results}
\end{figure}

\section{Conclusions}
\label{sec:5}

In this work, we present a unified four-field, total-pressure formulation for thermo- and multiple-network poroelasticity that admits robust mathematical analysis and stable discretizations. A fully discrete mixed finite element method is developed by employing Taylor–Hood elements with consistent stabilization. A global-in-time iterative decoupling scheme is proposed to reduce computational complexity by separating mechanics from fluid/thermal transport. We rigorously prove that the iteration is a contraction mapping independent of physical parameters, which yields unconditional convergence. Comprehensive numerical tests confirm the theoretical properties, demonstrating both stability and optimal convergence behavior in representative examples.

\backmatter





\bmhead{Acknowledgements}

The work of H. Gu is supported by the National NSF of China No. 123B2016 and Shenzhen University of Information Technology (Grant No. SUIT2025KJ008). 
M. Cai acknowledges the support of The University of Maryland Baltimore Life Science Discovery (UMBILD)
Accelerator and REACH Hub Award 1U01GM152511-03, and the I-GAP Small Tech Transfer Grant I-GAP OTT-STT-272, sponsored by the Office of Technology Transfer at Morgan State University. J. Li was partially supported by the Shenzhen Sci-Tech Fund No. JCYJ20250604144225032 and  RCJC20200714114556020, Guangdong Basic and Applied Research Fund No. 2023B1515250005,  National Center for Applied Mathematics Shenzhen, and SUSTech International Center for Mathematics.


\bibliography{sn-bibliography}

\begin{appendices}
\section{Proof of Lemma \ref{lem:truncation_error}}
\label{appendix}

\begin{proof}
For convenience, we introduce the residuals for a generic variable $\chi \in \{ \xi, \phi, \psi\}$ as:
\begin{align*}
    \rho_\chi^n :=  (\partial_t \chi^n - \bar{\partial}_t \chi_h^n) + \bar{\partial}_t e_\chi^{I,n}.
\end{align*}
Using Taylor expansion, the residuals satisfy the following estimate:
\begin{align}
    \| \rho_\chi^n \|_{L^2}^2 
    & \lesssim \Delta t \int_{t_{n-1}}^{t_n} \|\partial_{tt} \chi \|^2 \, ds + \tfrac{1}{\Delta t} \int_{t_{n-1}}^{t_n} \|\partial_t e_\chi^{I,n} \|^2 \, ds.
\end{align}

Subtracting the fully discrete scheme \eqref{discreteform} from the continuous variational formulation \eqref{weakform} at $t=t_n$, we obtain the following system governing the discrete errors.
\begin{subequations}
\label{subdiscreteform}
\begin{align}
2\mu(\varepsilon(\bm e_u^n),\varepsilon(\bm v_h))
- (e_\xi^n,\nabla\cdot\bm v_h)
& =  0, \label{subdis:mon1}
\\
(\nabla\cdot\bm e_u^n, w_h)
+\tfrac{1}{\lambda} (e_\xi^n, w_h)
-\tfrac{\alpha}{\lambda} (e_\phi^n, w_h)
-\tfrac{\beta}{\lambda} (e_\psi^n, w_h)
& = 0,  \label{subdis:mon2}
\\
(c_1+\tfrac{\alpha^2}{\lambda})( {\partial}_t \phi^n - \bar{\partial}_t \phi^n_h , q_h)
+(\tfrac{\alpha\beta}{\lambda}-b_0)( {\partial}_t \psi^n - \bar{\partial}_t \psi^n_h , q_h)
-\tfrac{\alpha}{\lambda}( {\partial}_t \xi^n - \bar{\partial}_t \xi^n_h , q_h & ) \nonumber \\ 
+ K(\nabla e_\phi^n, \nabla q_h) + \gamma (e_\phi^n-e_\psi^n, q_h)
- \eta_{\phi} h^2 ( \nabla \bar{\partial}_t \phi^n_h, \nabla q_h) & = 0,  \label{subdis:mon3}
\\
(c_2+\tfrac{\beta^2}{\lambda})({\partial}_t \psi^n - \bar{\partial}_t \psi^n_h, s_h)
+(\tfrac{\alpha\beta}{\lambda}-b_0)( {\partial}_t \phi^n - \bar{\partial}_t \phi^n_h , s_h)
-\tfrac{\beta}{\lambda}({\partial}_t \xi^n - \bar{\partial}_t \xi^n_h, s_h & ) 
\nonumber \\
+ D(\nabla e_\psi^n, \nabla s_h) + \gamma (e_\psi^n - e_\phi^n, s_h)
- \eta_{\psi} h^2 ( \nabla \bar{\partial}_t \psi^n_h , \nabla s_h)  & = 0.  \label{subdis:mon4}
\end{align}
\end{subequations}
Using \eqref{eq:stokes_projection}, we rewrite equations \eqref{subdis:mon1} and \eqref{subdis:mon2} as:
\begin{align}
& 2\mu(\varepsilon(\bm e_u^{h,n}),\varepsilon(\bm v_h))
- (e_\xi^{h,n},\nabla\cdot\bm v_h) 
=  0, \label{adj:mon1}
\\
& (\nabla\cdot  \bar{\partial}_t \bm e_u^{h,n}, w_h)
+\tfrac{1}{\lambda} ( \bar{\partial}_t e_\xi^{h,n}, w_h)
-\tfrac{\alpha}{\lambda} ( \bar{\partial}_t e_\phi^{h,n}, w_h)
-\tfrac{\beta}{\lambda} ( \bar{\partial}_t e_\psi^{h,n}, w_h) 
\nonumber \\
= & - \tfrac{1}{\lambda} ( \bar{\partial}_t e_\xi^{I,n}, w_h)
+ \tfrac{\alpha}{\lambda} ( \bar{\partial}_t e_\phi^{I,n}, w_h)
+ \tfrac{\beta}{\lambda} ( \bar{\partial}_t e_\psi^{I,n}, w_h) 
\nonumber \\
= & \ 
(\nabla \cdot{\partial}_t \bm u^n - \nabla \cdot \bar{\partial}_t \bm u^n, w_h) 
-\tfrac{1}{\lambda} ( \rho_\xi^n, w_h)  + \tfrac{\alpha}{\lambda} (\rho_\phi^n, w_h)
+ \tfrac{\beta}{\lambda} (\rho_\psi^n, w_h).  \label{adj:mon2}
\end{align}
Summing the equations \eqref{subdis:mon3}, \eqref{subdis:mon4} and using \eqref{eq:elliptic_projection}, we obtain:
\begin{align}
& (c_1+\tfrac{\alpha^2}{\lambda})( \bar{\partial}_t e_\phi^{h,n} , q_h)
+(\tfrac{\alpha\beta}{\lambda}-b_0)( \bar{\partial}_t e_\psi^{h,n} , q_h)
-\tfrac{\alpha}{\lambda}( \bar{\partial}_t e_\xi^{h,n} , q_h ) + \eta_{\phi} h^2 ( \nabla \bar{\partial}_t e_\phi^{h,n}, \nabla q_h)
\nonumber \\
& + (c_2+\tfrac{\beta^2}{\lambda})( \bar{\partial}_t e_\psi^{h,n} , s_h)
+(\tfrac{\alpha\beta}{\lambda}-b_0)( \bar{\partial}_t e_\phi^{h,n} , s_h)
-\tfrac{\beta}{\lambda}( \bar{\partial}_t e_\xi^{h,n} , s_h ) + \eta_{\psi} h^2 ( \nabla \bar{\partial}_t e_\psi^{h,n}, \nabla s_h)
\nonumber \\
& + K(\nabla e_\phi^{h,n}, \nabla q_h) 
+ D(\nabla e_\psi^{h,n}, \nabla s_h) + \gamma(e_\phi^{h,n} - e_\psi^{h,n}, q_h - s_h) 
\nonumber \\ 
= & - (c_1+\tfrac{\alpha^2}{\lambda})( \rho_\phi^n , q_h)
- (\tfrac{\alpha\beta}{\lambda}-b_0)( \rho_\psi^n , q_h) + \tfrac{\alpha}{\lambda}( \rho_\xi^n , q_h ) - \eta_{\phi} h^2 ( \nabla \rho_\phi^n + \nabla {\partial}_t \phi^n, \nabla q_h) 
\nonumber \\ 
& - (c_2+\tfrac{\beta^2}{\lambda})( \rho_\psi^n , s_h) - (\tfrac{\alpha\beta}{\lambda}-b_0)(\rho_\psi^n , s_h) 
+ \tfrac{\beta}{\lambda}(\rho_\xi^n, s_h ) 
- \eta_{\psi} h^2 ( \nabla \rho_\psi^n + \nabla {\partial}_t \psi^n, \nabla s_h),
\label{adj:diff}
\end{align}
We choose the test functions $\bm{v}_h = \bar{\partial}_t \bm{e}_u^{h,n}$ in \eqref{adj:mon1}, $w_h = e_\xi^{h,n}$ in \eqref{adj:mon2}, and $(q_h , s_h) = (e_\phi^{h,n}, e_\psi^{h,n})$ in \eqref{adj:diff}. Summing the resulting equations and recalling the energy notations $\mathcal{E}$ and $\mathcal{V}$ defined in \eqref{def:E_func} and \eqref{def:V_func}, we obtain:
\begin{align}
& \tfrac{1}{2\Delta t} \mathcal{E}(\bm e_u^{h,n}, e_\xi^{h,n}, e_\phi^{h,n}, e_\psi^{h,n}) - \tfrac{1}{2\Delta t}\mathcal{E}(\bm e_u^{h,n-1}, e_\xi^{h,n-1}, e_\phi^{h,n-1}, e_\psi^{h,n-1}) 
\nonumber \\
& + \tfrac{\Delta t}{2}\mathcal{E}(\bar{\partial}_t \bm e_u^{h,n}, \bar{\partial}_t e_\xi^{h,n}, \bar{\partial}_t e_\phi^{h,n}, \bar{\partial}_t e_\psi^{h,n}) + \mathcal{V}(e_\phi^{h,n}, e_\psi^{h,n}) 
\nonumber \\
& + \tfrac{\eta_\phi h^2}{2\Delta t}
(
\|\nabla e_\phi^{h,n}\|^2
-
\|\nabla e_\phi^{h,n-1}\|^2
) + \tfrac{\eta_\phi h^2 \Delta t }{2}
\|\nabla( \bar{\partial}_t e_\phi^{h,n})\|^2
\nonumber \\
& + \tfrac{\eta_\psi h^2}{2\Delta t}
(
\|\nabla e_\psi^{h,n}\|^2
-
\|\nabla e_\psi^{h,n-1}\|^2)
+ \tfrac{\eta_\psi h^2 \Delta t}{2}
\|\nabla( \bar{\partial}_t e_\psi^{h,n})\|^2
\nonumber \\
= & \; (\nabla \cdot{\partial}_t \bm u^n - \nabla \cdot \bar{\partial}_t \bm u^n, e_\xi^{h,n})
-\tfrac{1}{\lambda}
( \rho_\xi^n - \alpha \rho_\phi^n - \beta \rho_\psi^n,\;
      e_\xi^{h,n} - \alpha e_\phi^{h,n} - \beta e_\psi^{h,n} ) 
\nonumber \\
& 
-(c_1-b_0)(\rho_\phi^n,e_\phi^{h,n})
-(c_2-b_0)(\rho_\psi^n,e_\psi^{h,n})
- b_0(\rho_\phi^n-\rho_\psi^n,\; e_\phi^{h,n}-e_\psi^{h,n})
\nonumber \\
& 
- \eta_{\phi} h^2 ( \nabla \rho_\phi^n + \nabla {\partial}_t \phi^n, \nabla e_\phi^{h,n}) 
- \eta_{\psi} h^2 ( \nabla \rho_\psi^n + \nabla {\partial}_t \psi^n, \nabla e_\psi^{h,n}), \label{eq:energy_equality}
\end{align}
where we have used the discrete differentiation identity:
\[
( \bar{\partial}_t x^n, x^n ) =  \tfrac{1}{2\Delta t} \|x^n\|^2 - \tfrac{1}{2\Delta t} \|x^{n-1}\|^2 +  \tfrac{\Delta t}{2} \|x^n - x^{n-1}\|^2.
\]

We now proceed to estimate the right-hand side terms of \eqref{eq:energy_equality}. First, we apply the discrete inf-sup condition \eqref{disc:infsup} to bound the error term $e_\xi^{h,n}$ using \eqref{adj:mon1}.
\begin{align}
    \tilde{C}_{is} \| e_\xi^{h,n} \| & \leq
    \sup_{\bm 0 \not = \bm{v}_h \in \bm{V}_h} \frac{(\nabla \cdot \bm{v}_h, e_\xi^{h,n})}{\| \bm{v}_h \|_{H^1}} 
    \nonumber \\ 
    & = \sup_{\bm 0 \not =  \bm{v}_h \in \bm{V}_h} \frac{2\mu(\varepsilon(\bm e_u^{h,n}),\varepsilon(\bm v_h))}{\| \bm{v}_h \|_{H^1}}
    \leq 2\mu \| \varepsilon(\bm e_u^{h,n}) \|,    
\end{align}
implying \eqref{lem:te3}. Using this bound, along with the Cauchy-Schwarz and Young's inequalities, we estimate the primary residual terms as follows:
\begin{align}
& (\nabla\cdot\partial_t \bm{u}^n-\nabla\cdot\bar\partial_t \bm{u}^n, e_\xi^{h,n})
-\tfrac{1}{\lambda}(\rho_\xi^n-\alpha\rho_\phi^n-\beta\rho_\psi^n,\;
e_\xi^{h,n}-\alpha e_\phi^{h,n}-\beta e_\psi^{h,n})
\nonumber \\
& -(c_1-b_0)(\rho_\phi^n,e_\phi^{h,n})-(c_2-b_0)(\rho_\psi^n,e_\psi^{h,n}) - b_0(\rho_\phi^n-\rho_\psi^n,\; e_\phi^{h,n}-e_\psi^{h,n})
\nonumber \\
\le & \ \mu \| \varepsilon(\bm e_u^{h,n}) \|^2 + \tfrac{1}{2\lambda} \|e_\xi^{h,n}-\alpha e_\phi^{h,n}-\beta e_\psi^{h,n}\|^2 + \tfrac{c_1-b_0}{2}\|e_\phi^{h,n}\|^2 + \tfrac{c_2-b_0}{2}\|e_\psi^{h,n}\|^2
\nonumber \\
& + \tfrac{b_0}{2} \|e_\phi^{h,n} - e_\psi^{h,n}\|^2 + C \Delta t \int_{t_{n-1}}^{t_n} ( \|\nabla\cdot \partial_{tt} \bm{u}\|^2 + \|\partial_{tt} \xi\|^2 + \|\partial_{tt} \phi\|^2 + \|\partial_{tt} \psi\|^2) \, ds
\nonumber \\
& + C\tfrac{1}{\Delta t} \int_{t_{n-1}}^{t_n} ( \|\partial_t e_\xi^{I,n} \|^2 + \|\partial_t e_\phi^{I,n} \|^2 + \|\partial_t e_\psi^{I,n} \|^2) \, ds.
\label{est:primary_residuals}
\end{align}
Next, we treat the stabilization residuals as follows:
\begin{align}
& -\eta_\phi h^2(\nabla\rho_\phi^n+\nabla\partial_t\phi^n,\nabla e_\phi^{h,n}) 
- \eta_{\psi} h^2 ( \nabla \rho_\psi^n + \nabla {\partial}_t \psi^n, \nabla e_\psi^{h,n}) \nonumber \\
\le & \ \tfrac{ \eta_\phi h^2}{2}\|\nabla e_\phi^{h,n}\|^2 + \tfrac{\eta_\psi h^2}{2} \|\nabla e_\psi^{h,n}\|^2 + C \eta_\phi h^2\|\nabla\partial_t\phi^n\|^2 + C \eta_\psi h^2\|\nabla\partial_t\psi^n\|^2
\nonumber \\
& + C h^2 \Delta t \int_{t_{n-1}}^{t_n} (\eta_\phi  \| \nabla \partial_{tt} \phi\|^2 + \eta_\psi \| \nabla \partial_{tt} \psi\|^2) \, ds
\nonumber \\
& + C\tfrac{h^2}{\Delta t} \int_{t_{n-1}}^{t_n} ( \eta_\phi \| \nabla \partial_t e_\phi^{I,n} \|^2 + \eta_\psi \| \nabla \partial_t e_\psi^{I,n} \|^2) \, ds.
\label{est:stabilization_residuals}
\end{align}

Multiplying \eqref{eq:energy_equality} by $2\Delta t$, summing over $n=1$ to $N$, and dropping some non-negative terms on the left-hand side, we can combine \eqref{est:primary_residuals} and \eqref{est:stabilization_residuals} to obtain:
\begin{align}
& \mathcal{E}(\bm e_u^{h,N}, e_\xi^{h,N}, e_\phi^{h,N}, e_\psi^{h,N}) +  2 \Delta t \sum_{n=1}^N \mathcal{V}(e_\phi^{h,n}, e_\psi^{h,n}) 
\nonumber \\ 
\leq & \ \Delta t \sum_{n=1}^N \left(\mathcal{E}(\bm e_u^{h,n}, e_\xi^{h,n}, e_\phi^{h,n}, e_\psi^{h,n}) + \eta_\phi h^2 \|\nabla e_\phi^{h,n}\|^2 + \eta_\psi h^2 \|\nabla e_\psi^{h,n}\|^2 \right) 
\nonumber \\
& + C (\Delta t)^2 \int_{0}^{T_f} \big( \|\partial_{tt} \bm{u}\|_{H^1}^2 + \|\partial_{tt} \xi\|^2 + (1+\eta_{\phi}h^2) \|\partial_{tt} \phi\|^2_{H^1} + (1+\eta_{\psi}h^2)\|\partial_{tt} \psi\|^2_{H^1} \big) \, ds
\nonumber \\
& 
+ C \int_{0}^{T_f} \big( \|\partial_t e_\xi^{I} \|^2 + \|\partial_t e_\phi^{I} \|^2 + \|\partial_t e_\psi^{I} \|^2 + \eta_\phi h^2 \| \partial_t e_\phi^{I} \|_{H^1}^2 + \eta_\psi h^2 \| \partial_t e_\psi^{I} \|_{H^1}^2 \big) \, ds
\nonumber \\
& + C \Delta t \sum_{n=1}^N \big( \eta_\phi h^2\|\partial_t\phi^n\|^2_{H^1} + \eta_\psi h^2\|\partial_t\psi^n\|^2_{H^1} \big).
\label{ineq:energy_summed}
\end{align}
Applying the discrete Grönwall lemma and utilizing the approximation properties of the projection operators, we arrive at the final error estimate:
\begin{align}
& \mathcal{E}(\bm e_u^{h,N}, e_\xi^{h,N}, e_\phi^{h,N}, e_\psi^{h,N}) + 2 \Delta t \sum_{n=1}^N \mathcal{V}(e_\phi^{h,n}, e_\psi^{h,n})
\nonumber \\
\leq & \ C (\Delta t)^2 \int_{0}^{T_f} \big( \|\partial_{tt} \bm{u}\|_{H^1}^2 + \|\partial_{tt} \xi\|^2 + (1+\eta_{\phi}h^2) \|\partial_{tt} \phi\|^2_{H^1} + (1+\eta_{\psi}h^2)\|\partial_{tt} \psi\|^2_{H^1} \big) \, ds
\nonumber \\
& + C h^{2k} \int_{0}^{T_f} \big( \|\partial_t \bm u \|_{H^{k+1}}^2+ h^{2k}\|\partial_t \xi \|_{H^k}^2 \big) \, ds + C h^{2l+2} \int_{0}^{T_f} \big( \|\partial_t \phi \|_{H^{l+1}}^2 
+ \|\partial_t \psi \|_{H^{l+1}}^2 \big) \, ds
\nonumber \\
& + C T_f h^2 \big( \eta_\phi \max_{1\leq n \leq N} \| \partial_t\phi^n\|^2_{H^1}  + \eta_\psi \max_{1\leq n \leq N} \| \partial_t\psi^n\|^2_{H^1}  \big).
\label{ineq:energy_final}
\end{align}
Recalling that $2\mu\|\varepsilon(\bm{e}_u^{h,N})\|^2 \leq \mathcal{E}(\bm e_u^{h,N}, e_\xi^{h,N}, e_\phi^{h,N}, e_\psi^{h,N})$, we complete the proof of the first estimate.

To prove the second estimate, we rewrite the first equation of the error system as:
\begin{align}
2\mu(\varepsilon(\bar{\partial}_t \bm e_u^{h,n}),\varepsilon(\bm v_h))
- (\bar{\partial}_t e_\xi^{h,n},\nabla\cdot\bm v_h)
& =  0, \label{adj:mon12}
\end{align}
We choose the test functions $\bm v_h = \bar{\partial}_t \bm e_u^{h,n}$ in \eqref{subdis:mon1}, $w_h = \bar{\partial}_t e_\xi^{h,n}$ in \eqref{adj:mon2}, $q_h = \bar{\partial}_t e_\phi^{h,n}$ in \eqref{adj:diff}, and $s_h = \bar{\partial}_t e_\psi^{h,n}$ in \eqref{adj:diff}. This yields:
\begin{align}
& \mathcal{E}(\bar{\partial}_t \bm e_u^{h,n}, \bar{\partial}_t e_\xi^{h,n}, \bar{\partial}_t e_\phi^{h,n}, \bar{\partial}_t e_\psi^{h,n}) + \eta_\phi h^2 \|\nabla \bar{\partial}_t e_\phi^{h,n}\|^2 + \eta_\psi h^2 \|\nabla \bar{\partial}_t e_\psi^{h,n}\|^2 
\nonumber \\
& + \tfrac{1}{2\Delta t} \mathcal{V}( e_\phi^{h,n}, e_\psi^{h,n}) - \tfrac{1}{2\Delta t}\mathcal{V}( e_\phi^{h,n-1}, e_\psi^{h,n-1}) + \tfrac{\Delta t}{2}\mathcal{V}( \bar{\partial}_t e_\phi^{h,n}, \bar{\partial}_t e_\psi^{h,n})
\nonumber \\ 
= & \; (\nabla \cdot{\partial}_t \bm u^n - \nabla \cdot \bar{\partial}_t \bm u^n, \bar{\partial}_t e_\xi^{h,n})
-\tfrac{1}{\lambda}
( \rho_\xi^n - \alpha \rho_\phi^n - \beta \rho_\psi^n,\;
      \bar{\partial}_t e_\xi^{h,n} - \alpha \bar{\partial}_t e_\phi^{h,n} - \beta \bar{\partial}_t e_\psi^{h,n} ) 
\nonumber \\
& 
-(c_1-b_0)(\rho_\phi^n, \bar{\partial}_t e_\phi^{h,n})
-(c_2-b_0)(\rho_\psi^n, \bar{\partial}_t e_\psi^{h,n})
- b_0(\rho_\phi^n-\rho_\psi^n,\; \bar{\partial}_t e_\phi^{h,n} - \bar{\partial}_t e_\psi^{h,n})
\nonumber \\
& 
- \eta_{\phi} h^2 ( \nabla \rho_\phi^n + \nabla {\partial}_t \phi^n, \nabla \bar{\partial}_t e_\phi^{h,n}) 
- \eta_{\psi} h^2 ( \nabla \rho_\psi^n + \nabla {\partial}_t \psi^n, \nabla \bar{\partial}_t e_\psi^{h,n}), \label{eq:energy_equality3}
\end{align}
Similarly, we apply the discrete inf-sup condition \eqref{disc:infsup} and \eqref{adj:mon12} to bound
\begin{align}
    \tilde{C}_{is} \| \bar{\partial}_t e_\xi^{h,n} \| & \leq
    \sup_{\bm 0 \not = \bm{v}_h \in \bm{V}_h} \frac{(\nabla \cdot \bm{v}_h, \bar{\partial}_t e_\xi^{h,n})}{\| \bm{v}_h \|_{H^1}} 
    \nonumber \\ 
    & = \sup_{\bm 0 \not =  \bm{v}_h \in \bm{V}_h} \frac{2\mu(\varepsilon( \bar{\partial}_t \bm e_u^{h,n}),\varepsilon(\bm v_h))}{\| \bm{v}_h \|_{H^1}}
    \leq 2\mu \| \varepsilon(\bar{\partial}_t \bm e_u^{h,n}) \|.    
\end{align} 
We then estimate the primary residual terms:
\begin{align}
& (\nabla\cdot\partial_t \bm{u}^n-\nabla\cdot\bar\partial_t \bm{u}^n, \bar{\partial}_t e_\xi^{h,n})
-\tfrac{1}{\lambda}(\rho_\xi^n-\alpha\rho_\phi^n-\beta\rho_\psi^n,\;
\bar{\partial}_t e_\xi^{h,n}-\alpha \bar{\partial}_t e_\phi^{h,n}-\beta \bar{\partial}_t e_\psi^{h,n})
\nonumber \\
& -(c_1-b_0)(\rho_\phi^n,\bar{\partial}_t e_\phi^{h,n})-(c_2-b_0)(\rho_\psi^n,\bar{\partial}_t e_\psi^{h,n}) - b_0(\rho_\phi^n-\rho_\psi^n,\; \bar{\partial}_t e_\phi^{h,n}-\bar{\partial}_t e_\psi^{h,n})
\nonumber \\
\le & \ \mu \| \varepsilon(\bar{\partial}_t \bm e_u^{h,n}) \|^2 + \tfrac{1}{2\lambda} \|\bar{\partial}_t e_\xi^{h,n}-\alpha \bar{\partial}_t e_\phi^{h,n}-\beta \bar{\partial}_t e_\psi^{h,n}\|^2 + \tfrac{c_1-b_0}{2}\|\bar{\partial}_t e_\phi^{h,n}\|^2 + \tfrac{c_2-b_0}{2}\|\bar{\partial}_t e_\psi^{h,n}\|^2
\nonumber \\
& + \tfrac{b_0}{2} \|\bar{\partial}_t e_\phi^{h,n} - \bar{\partial}_t e_\psi^{h,n}\|^2 + C \Delta t \int_{t_{n-1}}^{t_n} ( \|\nabla\cdot \partial_{tt} \bm{u}\|^2 + \|\partial_{tt} \xi\|^2 + \|\partial_{tt} \phi\|^2 + \|\partial_{tt} \psi\|^2) \, ds
\nonumber \\
& + C\tfrac{1}{\Delta t} \int_{t_{n-1}}^{t_n} ( \|\partial_t e_\xi^{I,n} \|^2 + \|\partial_t e_\phi^{I,n} \|^2 + \|\partial_t e_\psi^{I,n} \|^2) \, ds.
\label{est:primary_residuals2}
\end{align}
Next, we treat the stabilization residuals:
\begin{align}
& -\eta_\phi h^2(\nabla\rho_\phi^n+\nabla\partial_t\phi^n,\nabla \bar{\partial}_t e_\phi^{h,n}) 
- \eta_{\psi} h^2 ( \nabla \rho_\psi^n + \nabla {\partial}_t \psi^n, \nabla \bar{\partial}_te_\psi^{h,n}) \nonumber \\
\le & \ \tfrac{ \eta_\phi h^2}{2}\|\nabla  \bar{\partial}_t e_\phi^{h,n}\|^2 + \tfrac{\eta_\psi h^2}{2} \|\nabla \bar{\partial}_t e_\psi^{h,n}\|^2 + C \eta_\phi h^2\|\nabla\partial_t\phi^n\|^2 + C \eta_\psi h^2\|\nabla\partial_t\psi^n\|^2
\nonumber \\
& + C h^2 \Delta t \int_{t_{n-1}}^{t_n} (\eta_\phi  \| \nabla \partial_{tt} \phi\|^2 + \eta_\psi \| \nabla \partial_{tt} \psi\|^2) \, ds
\nonumber \\
& + C\tfrac{h^2}{\Delta t} \int_{t_{n-1}}^{t_n} ( \eta_\phi \| \nabla \partial_t e_\phi^{I,n} \|^2 + \eta_\psi \| \nabla \partial_t e_\psi^{I,n} \|^2) \, ds.
\label{est:stabilization_residuals2}
\end{align}
Multiplying \eqref{eq:energy_equality3} by $2\Delta t$ and summing over $n=1$ to $N$, we drop the non-negative terms involving $\bar{\partial}_t$ on the left-hand side. Then, combining the residual estimates \eqref{est:primary_residuals2} and \eqref{est:stabilization_residuals2} with the approximation properties of the projection operators, we obtain:
\begin{align}
\Delta &  t \sum_{n=1}^N \mathcal{E}(\bar{\partial}_t \bm e_u^{h,n}, \bar{\partial}_t e_\xi^{h,n}, \bar{\partial}_t e_\phi^{h,n}, \bar{\partial}_t e_\psi^{h,n}) + \mathcal{V}(e_\phi^{h,N}, e_\psi^{h,N}) 
\nonumber \\ 
\leq & \ C (\Delta t)^2 \int_{0}^{T_f} \big( \|\partial_{tt} \bm{u}\|_{H^1}^2 + \|\partial_{tt} \xi\|^2 + (1 + \eta_\phi h^2) \| \partial_{tt} \phi\|_{H^1}^2 + (1 + \eta_\psi h^2) \| \partial_{tt} \psi\|_{H^1}^2\big) \, ds
\nonumber \\
& 
+ C \int_{0}^{T_f} \big( \|\partial_t e_\xi^{I} \|^2 + \|\partial_t e_\phi^{I} \|^2 + \|\partial_t e_\psi^{I} \|^2 + \eta_\phi h^2 \| \partial_t e_\phi^{I} \|_{H^1}^2 + \eta_\psi h^2 \| \partial_t e_\psi^{I} \|_{H^1}^2 \big) \, ds
\nonumber \\
& + C \Delta t \sum_{n=1}^N \big( \eta_\phi h^2\|\nabla\partial_t\phi^n\|^2 + \eta_\psi h^2\|\nabla\partial_t\psi^n\|^2 \big).
\nonumber \\
\leq & \ C (\Delta t)^2 \int_{0}^{T_f} \big( \|\partial_{tt} \bm{u}\|_{H^1}^2 + \|\partial_{tt} \xi\|^2 + (1+\eta_{\phi}h^2) \|\partial_{tt} \phi\|^2_{H^1} + (1+\eta_{\psi}h^2)\|\partial_{tt} \psi\|^2_{H^1} \big) \, ds
\nonumber \\
& + C h^{2k} \int_{0}^{T_f} \big( \|\partial_t \bm u \|_{H^{k+1}}^2+ h^{2k}\|\partial_t \xi \|_{H^k}^2 \big) \, ds + C h^{2l+2} \int_{0}^{T_f} \big( \|\partial_t \phi \|_{H^{l+1}}^2 
+ \|\partial_t \psi \|_{H^{l+1}}^2 \big) \, ds
\nonumber \\
& + C T_f h^2 \big( \eta_\phi \max_{1\leq n \leq N} \| \partial_t\phi^n\|^2_{H^1}  + \eta_\psi \max_{1\leq n \leq N} \| \partial_t\psi^n\|^2_{H^1}  \big).
\label{ineq:energy_final2}
\end{align}
Recalling that $K\|\nabla e_\phi^{h,N}\|^2 + D\|\nabla e_\psi^{h,N}\|^2 \leq \mathcal{V}(e_\phi^{h,N}, e_\psi^{h,N})$, we complete the proof of the second estimate.
This completes the proof.
\end{proof}

\end{appendices}




\end{document}